\documentclass[a4paper]{IEEEtran}
\usepackage{cite}

\ifCLASSINFOpdf
\else
\fi
\usepackage[cmex10]{amsmath}
\usepackage{array}

\ifCLASSOPTIONcompsoc
  \usepackage[caption=false,font=normalsize,labelfont=sf,textfont=sf]{subfig}
\else
  \usepackage[caption=false,font=footnotesize]{subfig}
\fi
\usepackage{url}

\usepackage{tabulary}
\usepackage[english]{babel}
\usepackage{blindtext}
\usepackage[makeroom]{cancel}
\usepackage{amsfonts}
\usepackage{amsthm}  
\usepackage{amssymb}
\usepackage[pdftex]{graphicx}
\graphicspath{{../pdf/}{../jpeg/}}
\DeclareGraphicsExtensions{.pdf,.jpeg,.png}
\usepackage{pgf,tikz}
\usepackage{tabularx}
\usepackage{float}

\usepackage{enumitem} 

\newcommand{\vect}[1]{\boldsymbol{#1}}
\newcommand{\mat}[1]{\boldsymbol{#1}}
\newcommand{\wt}[1]{\widetilde{#1}}

\renewcommand{\eqref}[1]{Eq.~(\ref{#1})}  

\DeclareMathOperator{\sign}{\mathrm{sign} \,}
\DeclareMathOperator{\degr}{\mathrm{deg}}
\DeclareMathOperator{\diag}{\mathrm{diag}}
\DeclareMathOperator{\intr}{\mathrm{Int}}

\newtheorem{remark}{Remark}
\newtheorem{theorem}{Theorem}
\newtheorem{lemma}{Lemma}

\newtheorem{definition}{Definition}

\newtheorem{proposition}{Proposition}

\begin{document}

\title{Applications of the Poincar\'e--Hopf Theorem: Epidemic Models and Lotka--Volterra Systems}
%
%
%

\author{\IEEEauthorblockN{Mengbin Ye, \emph{Member, IEEE} $\;$ }
\and
\IEEEauthorblockN{Ji Liu, \emph{Member, IEEE} $\; $ }
\and 
\IEEEauthorblockN{Brian D.O. Anderson, \emph{Life Fellow, IEEE} $\;$\\ }
\and 
\IEEEauthorblockN{Ming Cao, \emph{Senior Member, IEEE $\;$ }
}

\thanks{Ye and Cao were supported in part by the European Research Council (ERC-CoG-771687) and the Netherlands Organization for Scientific Research (NWO-vidi-14134). Anderson was supported by the Australian Research Council under grant \mbox{DP160104500} and DP190100887, and by Data61-CSIRO. 

M. Ye is with the Optus--Curtin Centre of Excellence in Artificial Intelligence, Curtin University, Perth, Australia.

M. Cao is with the Faculty of Science and Engineering, ENTEG, University of Groningen, Groningen 9747 AG, Netherlands. 

B.D.O.~Anderson is with the Research School of Electrical, Energy and Material Engineering, Australian National University, Canberra, Australia, and also with Hangzhou Dianzi University, China, and Data61-CSIRO in Canberra, Australia. 

J.~Liu is with the Department of Electrical and Computer Engineering, Stony Brook University.

\texttt{mengbin.ye@curtin.edu.au, m.cao@rug.nl, brian.anderson@anu.edu.au, ji.liu@stonybrook.edu}.}

}

\maketitle

\begin{abstract} This paper focuses on properties of equilibria and their associated regions of attraction for continuous-time nonlinear dynamical systems. The classical Poincar\'e--Hopf Theorem is used to derive a general result providing a sufficient condition for the existence of a unique equilibrium for the system. The condition involves the Jacobian of the system at possible equilibria, and ensures the system is in fact locally exponentially stable. We show how to apply this result to the deterministic susceptible-infected-susceptible (SIS) networked model, and a nonlinear Lotka--Volterra system. We use the result further to extend the SIS model via the introduction of a broad class of decentralised feedback controllers, which significantly change the system dynamics, rendering existing Lyapunov-based approaches to analysis of the system invalid. Using the Poincar\'e--Hopf based approach, we identify a necessary and sufficient condition under which the controlled SIS system has a unique nonzero equilibrium (a diseased steady-state), and we show using monotone systems theory that this nonzero equilibrium is attractive for all nonzero initial conditions.
A counterpart sufficient condition for the existence of a unique equilibrium for a nonlinear discrete-time dynamical system is also presented.
\end{abstract}

\begin{IEEEkeywords}
complex networks, differential topology, feedback control, monotone systems
\end{IEEEkeywords}

%
\IEEEpeerreviewmaketitle

\section{Introduction}\label{sec:intro}
%
%
%
%

\IEEEPARstart{M}{any} dynamical processes in the natural sciences can be studied as continuous-time systems of the form 
\begin{equation}\label{eq:general_system}
\dot{x} = f(x)
\end{equation}
where $f(\cdot)$ is a suitably smooth nonlinear vector-valued function, and $x = [x_1, \hdots, x_n]^\top$ represents a vector of $n\geq 1$ biological, chemical, or physical variables. In the course of conducting analysis on such models, it is often of interest to characterise the equilibria of \eqref{eq:general_system}, including the number, stability properties and associated regions of attraction. In context, there is usually (but not always) an equilibrium at $x = \vect 0_n$, where $\vect 0_n$ is the $n$-dimensional vector of all zeros, reflecting the situation where the modelled process has ceased completely, and we call it the trivial equilibrium. There is obvious interest to determine if there exist non-trivial equilibria of \eqref{eq:general_system}, and how many. One particular focus may be to determine conditions on $f$ such that \eqref{eq:general_system} has a \textit{unique} non-trivial equilibrium (if in fact any such conditions exist).


Suppose that one has an intuition perhaps obtained from extensive simulations that the particular \eqref{eq:general_system} system of interest has a unique non-trivial equilibrium, call it $x^*$. Then, \textit{the existence and uniqueness} of $x^*$ might be proved by analysis using algebraic calculations involving the particular $f(\cdot)$ of interest. If $f(\cdot)$ is highly nonlinear, or $n$ is large (e.g. \eqref{eq:general_system} is modelling a complex networked system), a proof of the uniqueness of $x^*$ reliant on the algebraic form of the specific $f$ may be extremely complicated. Some systems admit Lyapunov functions that simultaneously establish that $x^*$ is unique and that it is globally attractive~\cite{shuai2013epidemic_lyapunov}; this approach was also applied to several classes of coupled systems of differential equations over networks~\cite{li2010global_network}. However, such an approach is not applicable for systems, including those in the natural sciences, that exhibit limit cycles or chaos~\cite{takeuchi1996global,smale1976differential}.

Moreover, one may wish to modify some \eqref{eq:general_system} system by introducing additional nonlinearities, and obtain a new system $\dot{x} = \bar f(x)$. For example, and as we shall do in this paper, one may insert a feedback control $u(x)$ to ensure the closed-loop system $\dot{x} = f(x) + u(x)$ achieves some control objective. Alternatively, $f(\cdot)$ may have been obtained by making idealised assumptions of the process being modelled, and one wishes to relax or change these assumptions to better reflect the real world, resulting in a new system. Suppose that one were again interested in determining whether $\dot{x} = \bar f(x)$ had a unique non-trivial equilibrium, call it $\bar x^*$. A logical approach would be to extend the analysis method for the unique equilibrium $x^*$ for \eqref{eq:general_system} to consider $\bar{f}(\cdot)$. However, approaches relying on algebraic calculations using the specific $f(\cdot)$ may not be general enough to guarantee successful adaptation for the various modified $\bar{f}(x)$. Moreover, a Lyapunov function that works for \eqref{eq:general_system} may not work with $\dot{x} = \bar f(x)$, and finding a new Lyapunov function may prove challenging. 

Motivated by the above observations, this paper seeks to identify sufficient conditions for a general system of the form of  \eqref{eq:general_system} to have a unique equilibrium, involving as few calculations of the specific $f(\cdot)$ as possible. Once the existence of a unique equilibrium has been established, Lyapunov or other dynamical systems theory tools (as will be the case in this paper) can be used to identify regions of convergence.

\subsection{Contributions of This Paper}

There are several contributions of this paper, which we now detail. First, we use the classical Poincar\'e--Hopf Theorem \cite{milnor1997topology} from differential topology to derive a sufficient condition that simultaneously establishes the existence \textit{and} uniqueness of the equilibrium for a general nonlinear system \eqref{eq:general_system}, and that the equilibrium is locally exponentially stable. One can consider our result to be a specialisation of the Poincar\'e--Hopf Theorem. No conclusions are drawn on the existence or nonexistence of limit cycles or chaotic behaviour, though additional tools described later in the paper can establish such conclusions. Some existing works have used the Poincar\'e--Hopf Theorem to count equilibria, but typically focus on a specific system of interest within a specific applications domain (including sometimes static as opposed to dynamical systems) \cite{belabbas2013_formationstable,moghadas2004_PHepidemics,simsek2007generalized_PH,tang2007_PHcongestion,christensen2017PH_potentialgames,henshaw2015_SVIR_PoincareHopf,konovalov2010price_PH,rijk1983equilibrium,varian1975ph_unique}. We then apply the result to three example systems from the natural sciences. While these example systems are all positive systems, i.e. $x_i(t) \geq 0$ for all $i = 1, \hdots, n$ and $t \geq 0$, our result can be applied to many general nonlinear systems, with no restriction on the signs of the states.

Key to our approach is to check whether the Jacobian of $f(\cdot)$ in \eqref{eq:general_system} at every possible equilibrium is stable, though no a priori knowledge is needed that an equilibrium even exists. While computation of the Jacobian does require some knowledge of the algebraic form of $f(\cdot)$, we have found that in applying our approach to established models of biological systems, the level of complexity in calculations based on the specific algebraic form of $f(\cdot)$ is significantly reduced. 

The first example system we study is the deterministic Susceptible-Infected-Susceptible (SIS) network model for an epidemic spreading process. There is a well known necessary and sufficient condition for the SIS model to have a unique non-trivial \textit{endemic} equilibrium (which corresponds to the disease being present in the network) in addition to the trivial \textit{healthy} equilibrium (which corresponds to a disease-free network) \cite{lajmanovich1976SISnetwork,fall2007SIS_model,khanafer2016SIS_positivesystems,mei2017epidemics_review,vanMeighem2009_virus}. 
We show how the existence and uniqueness of this endemic equilibrium, based on this known condition, can be easily established using our aforementioned specialisation of the Poincar\'e--Hopf Theorem.

Next, we introduce decentralised feedback controllers into the SIS network model as our second example, with the objective of globally stabilising the controlled SIS network to the healthy equilibrium. The equations for the controlled system are no longer quadratic, as they were for the uncontrolled system and existing approaches, including those based on Lyapunov theory, cannot be extended to consider the controlled system.  
In contrast, we show that the Poincar\'e--Hopf based approach admits a direct and rather straightforward extension from the uncontrolled SIS system to the controlled SIS system. This allows us to prove that the controlled system has a unique endemic equilibrium, which is locally exponentially stable, if and only if the uncontrolled system has a unique endemic equilibrium.
We then appeal to results from monotone systems theory \cite{smith1988monotone_survey,smith2008monotone_book} to prove that the unique endemic equilibrium is in fact asymptotically stable for all feasible nonzero initial conditions. 
Our analysis covers a broad class of controllers, significantly extending a special case in \cite{liu2019bivirus}.

Last, we apply the specialisation of the Poincar\'e--Hopf Theorem to generalised nonlinear Lotka--Volterra systems first studied in \cite{goh1978sector}, which are popular for modelling the interaction of populations of biological species \cite{takeuchi1996global}. We use the Poincar\'e--Hopf approach to relax the sufficient condition of \cite{goh1978sector} for ensuring the existence of a unique  non-trivial equilibrium (and establish that it is locally exponentially stable). Limit cycles and chaotic behaviour, arising in many Lotka--Volterra systems, are not ruled out. Taking the same condition as in \cite{goh1978sector}, we then recover the global convergence result of \cite{goh1978sector} but with a simplified argument.
	
Naturally, one may also wish to consider nonlinear discrete-time  systems $x(k) = G(x(k)), \,k = 0, 1, \hdots$. It turns out that there is a counterpart condition for establishing existence and uniqueness of the equilibrium, which was first reported in \cite{anderson2017_lefschetzECC_arxiv}, and is established using the Lefschetz--Hopf Theorem \cite{hirsch2012differential}. In this paper, we recall the discrete-time result of \cite{anderson2017_lefschetzECC_arxiv} and its application to the DeGroot--Friedkin model of a social network \cite{jia2015opinion_SIAM}, and compare it against the result we derived for \eqref{eq:general_system}.
	
A preliminary version of this paper has been accepted in the 21st IFAC World Congress \cite{ye2020_IFAC_impossible}, covering limited results on the controlled SIS network model. This paper provides more material on the Poincar\'e--Hopf Theorem specialisation and its motivations, development of monotone systems theory, results on generalised Lotka--Volterra systems, and discussion of the discrete-time counterpart.

The rest of the paper is structured as follows. In Section~\ref{section:background_problem}, we provide relevant mathematical notation and preliminaries, and an explicit motivating example with the network SIS model. Section~\ref{sec:PH_general} introduces the Poincar\'e--Hopf Theorem, and the specialisation for application to general nonlinear systems. This specialisation is applied to the network SIS model in Section~\ref{sec:epidemics}, and Lotka--Volterra models in Section~\ref{sec:other}. The discrete-time result is covered in Section~\ref{sec:DT}, and conclusions are given in Section~\ref{sec:conclusion}.

\section{Background and Preliminaries}\label{section:background_problem}

\subsection{Notation}\label{ssec:notation}
To begin, we establish some mathematical notation. The $n$-column vector of all ones and zeros is given by $\vect{1}_n$ and $\vect{0}_n$, respectively. The $n\times n$ identity and $n\times m$ zero matrices are given by $I_n$ and $\mat 0_{n\times m}$, respectively. For a vector $a$ and matrix $A$, we denote the $i^{th}$ entry of $a$ and $(i,j)^{th}$ entry of $A$ as $a_i$ and $a_{ij}$, respectively. For any two vectors $a, b\in\mathbb{R}^n$, we write $a\geq  b$ and $a > b$ if $a_{i}\geq b_{i}$ and $a_{i}> b_{i}$, respectively, for all $i\in \{1, \hdots, n\}$. 
A real matrix $A \in \mathbb{R}^{n\times m}$ is said to be nonnegative or positive if $A \geq \mat 0_{n\times m}$ or $A > \mat 0_{n\times m}$, respectively. 

For a real square matrix $M$ with spectrum $\sigma(M)$, we use $\rho(M) = \max \left\{|\lambda|\ : \ \lambda\in\sigma(M)\right\}$ and 
$s(M) = \max \left\{{\rm Re}(\lambda)\ : \ \lambda\in\sigma(M)\right\}$ to denote the spectral radius of $M$ and the largest real part among the eigenvalues of $M$, respectively. A matrix $M$ is said to be \textit{Hurwitz} if $s(M) < 0$.

The Euclidean norm is $\Vert \cdot \Vert$, and the $(m-1)$-dimensional sphere embedded in $\mathbb{R}^m$ is denoted by $\mathcal{S}^{m-1}$. For a set $\mathcal{M}$ with boundary, we denote the boundary as $\partial\mathcal{M}$, and the interior $\intr(\mathcal{M}) \triangleq \mathcal{M}\setminus \partial\mathcal{M}$. We define the set 
\begin{equation}\label{eq:Xi}
	\Xi_n = \{x \in \mathbb{R}^n : 0 \leq x_i \leq 1, i \in \{1, \hdots, n\}\},
\end{equation}
and denote by $\mathbb{R}^n_{\geq 0} = \{x \in \mathbb{R}^n : x_i \geq 0\,,\forall\,i=1, \hdots, n\}$ and $\mathbb{R}^n_{> 0}=  \{x \in \mathbb{R}^n: x_i > 0\,,\forall\,i=1, \hdots, n\}$ the positive orthant and the interior of the positive orthant, respectively.

\subsection{Graph Theory}
For a directed graph $\mathcal{G} = (\mathcal{V}, \mathcal{E}, A)$, $\mathcal{V} = \{1, \hdots, n\}$ is the set of vertices (or nodes). The set of directed edges is given by $\mathcal{E}\subseteq \mathcal{V} \times \mathcal{V}$ and the edge $(i, j)$ is an arc that is incoming with respect to $j$ and outgoing with respect to $i$. The  matrix $A$ is defined such that $(i, j) \in \mathcal{E}$ if and only if $a_{ji} \neq 0$. We will sometimes write ``the matrix $A$ associated with $\mathcal{G}$'', or write $\mathcal{G}[A]$ to represent $\mathcal{G} = (\mathcal{V}, \mathcal{E}, A)$. We define the neighbour set of $i$ as $\mathcal{N}_i \triangleq \{j : e_{ji} \in \mathcal{E}\}$.
A directed path is a sequence of edges $(p_1, p_2), (p_2, p_3), ...,$ where $p_i \in \mathcal{V}$ are distinct and $(p_{i}, p_{i+1}) \in \mathcal{E}$. A graph $\mathcal{G}[A]$ is \emph{strongly connected} if and only if there is a path from every node to every other node, which is equivalent to $A$ being \emph{irreducible} \cite{berman1979nonnegative_matrices}.

\subsection{A Motivating Example: The Network SIS Model}\label{ssec:SIS_intro}
To more explicitly motivate the application of the Poincar\'e--Hopf Theorem, we introduce the first of several examples studied in this paper, viz. the network Susceptible-Infected-Susceptible (SIS) model \cite{lajmanovich1976SISnetwork}, which is a fundamental model in the deterministic epidemic modelling literature. To remain concise, we do not discuss the modelling derivations for which details are found in e.g. \cite{lajmanovich1976SISnetwork,nowzari2016epidemics}.

For some disease of interest, it is assumed that each individual is either Infected (I) with the disease, or is Susceptible (S) but not infected, and the individual can transition between the two states. Each individual resides in a population of large and constant size, and there is a metapopulation network of such populations, captured by a graph $\mathcal{G} = (\mathcal{V}, \mathcal{E}, B)$ with $n \geq 2$ nodes, where each node represents a population. Associated with $i \in \mathcal{V}$ is the variable $x_i(t) \in [0, 1]$, which represents the proportion of population $i$ that is Infected (and thus $1-x_i(t)$ represents the proportion of population $i$ that is Susceptible). The SIS dynamics for $x_i(t)$ are given by 
\begin{equation}\label{eq:SIS_contact_node}
\dot{x}_i(t) = - d_i x_i(t) + (1-x_i(t)) \sum_{j \in \mathcal{N}_i} b_{ij} x_j(t),
\end{equation}
where $d_i > 0$ is the recovery rate of node $i$, and for a node $j$ that is a neighbour of node $i$, i.e. $j\in \mathcal{N}_i$, $b_{ij} > 0$ is the infection rate from node $j$ to node $i$. If $j \notin \mathcal{N}_i$, then $b_{ij} = 0$. Defining $x = [x_1, \hdots, x_n]^\top \in \mathbb{R}^n$, one obtains 
\begin{equation}\label{eq:SIS_contact_network}
\dot{x}(t) = (- D  + B -X(t)B)x(t),
\end{equation}
with $X(t) = \diag(x_1(t), \hdots, x_n(t))$, $D = \diag(d_1, \hdots, d_n)$ being diagonal matrices. The matrix $B \geq \mat 0_{n\times n}$ is associated with the graph $\mathcal{G}$. With $\Xi_n$ defined in \eqref{eq:Xi}, and under the intuitively reasonable assumption that $x(0)\in \Xi_n$, one can prove that $x(t) \in \Xi_n$ for all $t \geq 0$, which means the dynamics in \eqref{eq:SIS_contact_network} are well defined and $x(t)$ retains its important physical meaning for all $t \geq 0$. Thus,  $\Xi_n$ is considered as the set of feasible initial conditions for \eqref{eq:SIS_contact_network}.

Obviously, $\vect 0_n$ is an equilibrium of \eqref{eq:SIS_contact_network}, termed the \textit{healthy} (or trivial) equilibrium. Any other equilibrium $x^* \in \Xi_n\setminus \vect 0_n$ is an \textit{endemic} equilibrium, as the disease persists in at least one node. The following result completely characterises the equilibria and the limiting behaviour of \eqref{eq:SIS_contact_network}.
\begin{proposition}[\hspace{-0.3pt}\cite{lajmanovich1976SISnetwork,fall2007SIS_model}]\label{prop:SIS_network_convergence}
Consider \eqref{eq:SIS_contact_network}, and suppose that $\mathcal{G}[B]$ is strongly connected. With $\Xi_n$ defined in \eqref{eq:Xi}, the following hold
\begin{enumerate}
	\item If $s(-D + B) \leq 0$, then $x =\vect 0_n$ is the unique equilibrium of \eqref{eq:SIS_contact_network}, and $\lim_{t\to\infty} x(t) = \vect 0_n$ for all $x(0) \in \Xi_n$.
	\item If $s(-D + B) > 0$, then in addition to the equilibrium $x =\vect 0_n$, there is a unique endemic equilibrium $x^* \in \intr(\Xi_n)$. Moreover, $\lim_{t\to\infty} x(t) = x^*$ for all $x(0) \in \Xi_n\setminus \vect 0_n$.
\end{enumerate}
\end{proposition}
A key conclusion of Proposition~\ref{prop:SIS_network_convergence} is that there is an endemic equilibrium $x^* \in \intr(\Xi_n)$ if and only if $s(-D+B) > 0$, and then in fact $x^*$ is the unique endemic equilibrium. In some approaches, the uniqueness of $x^*$ is first proved before construction of a Lyapunov function to establish convergence to $x^*$~\cite{fall2007SIS_model,lajmanovich1976SISnetwork,mei2017epidemics_review,khanafer2016SIS_positivesystems}. Other works construct a different Lyapunov function that can prove uniqueness of $x^*$ and convergence simultaneously~\cite{shuai2013epidemic_lyapunov}. In Section~\ref{sec:epidemics}, we will illustrate the effectiveness of the proposed approach of this paper, which relies on the Poincar\'e--Hopf Theorem. Importantly, we show in Section~\ref{ssec:feedback} how our method can be easily extended to study a modified version of \eqref{eq:SIS_contact_node} that incorporates decentralised feedback control, while there are no obvious paths to extend the aforementioned Lyapunov-based methods.
%

\section{Application of the Poincar\'e-Hopf Theorem For A Class of Nonlinear Systems}\label{sec:PH_general}


Before introducing the main result of this section, which is one of the key novel contributions of this paper, we first detail the notion of a \textit{tangent cone}~\cite{blanchini1999set_invariance}, and define what is meant by a vector field ``pointing inward'' to a set $\mathcal{M}$, and introduce the relevant aspects of differential topology.
	
Let the distance between a point $x\in \mathbb{R}^n$ and a compact set $\mathcal{Q} \subset \mathbb{R}^n$ be defined as
\begin{equation}
	\text{dist}(x, \mathcal{Q}) = \inf_{y\in\mathcal{Q}} \Vert x - y \Vert.
\end{equation} 
The tangent cone to $\mathcal{Q}$ at $x$ is the set
\begin{equation}\label{eq:tangent_cone}
	\mathcal{Z}_{\mathcal{Q}}(x) = \left\{ z \in \mathbb{R}^n : \lim_{\hspace{3pt}h\to 0}  \inf \frac{\text{dist}(x + hz, \mathcal{Q})}{h} = 0\right\}.
\end{equation}
If $\mathcal{Q}$ has a boundary $\partial\mathcal Q$, the one-sided limit $\lim_{h\to 0^+}$ is used in \eqref{eq:tangent_cone} for $x\in \partial Q$, see~\cite[Appendix D]{ye2019_PH_submit}. This paper will consider \eqref{eq:general_system} on an $m$-dimensional compact manifold $\mathcal{M} \subset \mathbb{R}^n$, with $m \leq n$. It is important to distinguish the dimension of $\mathcal{M}$, viz. $m$, from the dimension of the ambient space $\mathbb{R}^n$ in which $\mathcal{M}$ is embedded, viz. $n$. A natural choice is to embed $\mathcal{M}$ into a space of the dimension of $f$ in \eqref{eq:general_system}, and which we will assume henceforth. For example, if $\mathcal{M}$ is a ball, then clearly $m = n$. If, however, $\mathcal{M}$ is the $(n-1)$-sphere embedded in $\mathbb{R}^n$, then $m = n - 1< n$.

We now relate $\mathcal{Z}_{\mathcal{M}}(x)$ to the tangent space of $\mathcal{M}$ at $x$, denoted by $T_{x}\mathcal{M} \subset \mathbb{R}^m$. For all $y \notin \mathcal{M}$, one has $\mathcal{Z}_{\mathcal{M}}(y) = \emptyset$. However, $\mathcal{Z}_{\mathcal{M}}(x) = T_{x}\mathcal{M} \subset \mathbb{R}^m$ for all $x \in \intr(\mathcal{M})$. That is, the tangent cone at any point $x$ in the interior of $\mathcal{M}$ is equal to the tangent space of $\mathcal{M}$ at the same point, being a Euclidean space with the same dimension as $\mathcal{M}$. If $\mathcal{M}$ has a boundary, then $\mathcal{Z}_{\mathcal{M}}(x) \subset T_{x}\mathcal{M}$ for all $x \in \partial\mathcal{M}$. More specifically, the tangent cone on a boundary point is a subset of the tangent space, comprised of vectors whose directions ``point inward'' to $\mathcal{M}$ (as detailed below). These conclusions are intuitive from a geometric viewpoint, and proved (with additional details) in Appendix~\ref{app:tangent_cone}.
Armed with this knowledge, we now recall the following classical result, and a definition that will be useful in the sequel.
\begin{proposition}[\hspace{-0.3pt}{Nagumo's Theorem \cite[Theorem 3.1]{blanchini1999set_invariance}}]\label{prop:nagumo}
	Consider the system in \eqref{eq:general_system}, and suppose that it has a globally unique solution for every initial condition. Let $\mathcal{M} \subset \mathbb{R}^n$ be an $m$-dimensional compact and smooth manifold, with $m\leq n$. Then, $\mathcal{M}$ is positively invariant for the system if and only if $f(x) \in \mathcal{Z}_{\mathcal{M}}(x)$ for all $x\in\mathcal{M}$.
\end{proposition}

Let $f : U \to \mathbb{R}^n$ be a vector-valued function, where $U \subseteq \mathbb{R}^n$. On a manifold $\mathcal{M}$ of appropriate dimension, a vector field can be represented by $f$, as the mapping $f : \mathcal{M} \to T\mathcal{M}$, where $\mathcal{M} \subseteq U$ and $T\mathcal{M} \subseteq \mathbb{R}^n$.
\begin{definition}[Pointing inward]\label{def:inward}
	Consider a vector field defined by $f : \mathcal{M} \to T\mathcal{M}$, where $\mathcal{M} \subset \mathbb{R}^n$ is an $m$-dimensional compact manifold with boundary $\partial \mathcal{M}$, and $m \leq n$. The vector field is said to point inward to $\mathcal{M}$ at a point $x \in \mathcal{M}$ if
	\begin{equation}\label{eq:point_inward}
		f(x) \in \mathcal{Z}_{\mathcal{M}}(x)\setminus \partial \mathcal{Z}_{\mathcal{M}}(x).
	\end{equation}
\end{definition}
A vector field defined by a vector-valued function $g$ is said to point outward if $f = -g$ satisfies \eqref{eq:point_inward}. Since the vector field is represented by $f$ on $\mathcal{M}$, the phrases ``the vector field points inward'' and ``$f$ points inward'' will in this paper connote the same thing, as defined in Definition~\ref{def:inward}.

\subsection{Algebraic and Differential Topology}\label{ssec:diff_top}

We now introduce some definitions and concepts from topology, and then recall the Poincar\'e-Hopf Theorem. To stay focused on applications to existing models, we do not provide extensive details, which can be found~\cite{milnor1997topology,guillemin2010differential}. 

Consider a smooth map $f : X \to Y$, where $X$ and $Y$ are manifolds. Then, associated with $f$ at a point $x \in X$ is a linear derivative mapping $df_x : T_xX \to T_{f(x)}Y$, where $T_xX$ and $T_{f(x)}Y$ are the tangent space of $X$ at $x \in X$ and $Y$ at $y = f(x) \in Y$, respectively. If the manifold $X$ locally at $x$ looks like $\mathbb{R}^m$, then $df_x$ is simply the Jacobian of $f$ evaluated at $x$ in a local coordinate basis. Suppose that $X$ and $Y$ are of the same dimension. A point $x\in X$ is called a regular point if $df_x$ is nonsingular, and a point $y\in Y$ is called a regular value if $f^{-1}(y)$ contains only regular points. 

Suppose further that $X$ and $Y$ are manifolds of the same dimension without boundary, with $X$ compact and $Y$ connected. The (Brouwer) degree of $f$ at a regular value $y \in Y$ is \cite{milnor1997topology}
\begin{equation}\label{eq:deg_definition}
\degr(f,y) = \sum_{x\in f^{-1}(y)} \sign \det(df_x).
\end{equation}
Here, $\det(df_x)$ is the determinant of $df_x$, and $\sign \det(df_x) = \pm 1$ is simply the sign of the determinant of $df_x$ (note that $y$ being a regular value implies $df_x$ is nonsingular). Notice that $\sign \det(df_x)$ is $+1$ or $-1$ according as $df_x$ preserves or reverses orientation. Remarkably, $\degr(f,y)$ is independent of the choice of regular value $y$ \cite[Theorem A]{milnor1997topology}, and we can thus write the left hand side of \eqref{eq:deg_definition} simply as $\degr(f)$. 

A point $x \in X$ is said to be a zero of $f$ if $f(x) = \vect 0$, and we say that a zero $x$ is isolated if there exists an open ball around $x$ which contains no other zeros. A zero $x$ with nonsingular $df_x$ is said to be nondegenerate, and nonsingularity of $df_x$ is a sufficient condition for $x$ to be isolated. For an isolated zero $x$ of $f$, pick a closed ball $\mathcal{D}$ centred at $x$ such that $x$ is the only zero of $f$ in $\mathcal{D}$. The index of $x$, denoted $\text{ind}_x(f)$, is defined to be the degree of the map
\begin{align*}
u & : \partial\mathcal{D} \to \mathcal{S}^{m-1} \\
z & \mapsto \frac{f(z)}{\Vert f(z)\Vert}.
\end{align*}
If $x$ is a nondegenerate zero, then $\degr(u) = \sign \det(df_x)$ \cite[Lemma 4]{milnor1997topology}. 

Last, for a topological space $X$, we introduce the Euler characteristic $\chi(X)$ \cite{guillemin2010differential,milnor1997topology}, an integer number associated\footnote{While the Euler characteristic can be extended to noncompact $X$, this paper will only consider the Euler characteristic for compact $X$.} with $X$. A key property is that distortion or bending of $X$ (specifically a homotopy) leaves the number invariant. Euler characteristics are known for a great many topological spaces. 

While variations of the Poincar\'e--Hopf Theorem exist, with subtle differences, we now state one which will be sufficient for our purposes.
\begin{proposition}[The Poincar\'e-Hopf Theorem\cite{milnor1997topology}]\label{thm:ph}
	Consider a smooth vector field on a compact $m$-dimensional manifold $\mathcal{M}$, defined by the map $f : \mathcal{M} \to T\mathcal{M}$. If $\mathcal{M}$ has a boundary $\partial\mathcal{M}$, then $f$ must point outward at every point on $\partial\mathcal{M}$. Suppose that every zero $x_i \in \mathcal{M}$ of $f$ is nondegenerate. Then,
	\begin{equation}\label{eq:poincare_hopf_sum}
	\sum_{i} \mathrm{ind}_{x_i}(f) = \sum_i \sign \det(df_{x_i}) = \chi(\mathcal{M}),
	\end{equation}
	where $\chi(\mathcal{M})$ is the Euler characteristic of $\mathcal{M}$.
\end{proposition}

\subsection{Uniqueness of Equilibrium for General Nonlinear Systems}\label{ssec:general_ph}

A specialisation of the Poincar\'e--Hopf Theorem will  now be presented, which will be applied to different established dynamical models in Sections~\ref{sec:epidemics} and \ref{sec:other}.

We focus on the system \eqref{eq:general_system} on contractible manifolds. A manifold $\mathcal{M}$ is contractible if it is homotopy equivalent to a single point, or roughly speaking, $\mathcal{M}$ can be continuously deformed and shrunk into a single point. Any compact and convex subset of $\mathbb{R}^n$ is contractible, e.g. $\Xi_n$ in \eqref{eq:Xi}. A contractible manifold $\mathcal{M}$ has Euler characteristic $\chi(\mathcal{M}) = 1$. The following is one of the paper's main novel contributions.

\begin{theorem}[Unique Equilibrium]\label{thm:unique}
	Consider the autonomous system 
	\begin{equation}\label{eq:system}
	\dot{x} = f(x)
	\end{equation}
	where $f$ is smooth, and $x \in \mathbb{R}^n$. Let $\mathcal{M} \subset \mathbb{R}^n$, be an $m$-dimensional compact, contractible, and smooth manifold with boundary $\partial\mathcal{M}$, and with $m\leq n$. Suppose that $\mathcal{M}$ is positively invariant for \eqref{eq:system} and furthermore, $f$ points inward\footnote{Note that this implies that $\mathcal{M}$ cannot have equilibria on its boundaries.} to $\mathcal{M}$ at every point $x\in \partial\mathcal{M}$. If $df_{\bar x}$ is Hurwitz for every $\bar x \in \mathcal{M}$ satisfying $f(\bar x) = 0$, then \eqref{eq:system} has a unique equilibrium $x^* \in \intr(\mathcal{M})$. Moreover, $x^*$ is locally exponentially stable.
\end{theorem}
\begin{proof}
	The bulk of the proof focuses on establishing the properties of \eqref{eq:system} which will allow the existence and uniqueness of the equilibrium $x^*$ to be concluded from application of the Poincar\'e-Hopf Theorem, viz. Proposition~\ref{thm:ph}.
	
	To begin, we need to connect the language of Proposition~\ref{thm:ph} to that of Theorem~\ref{thm:unique}. First, recall the identification of the tangent cone for $\mathcal{M}$ immediately above Proposition~\ref{prop:nagumo}. The theorem assumptions, in conjunction with Definition~\ref{def:inward} and Proposition~\ref{prop:nagumo}, imply that  for all $x\in \mathcal{M}$, $f(x)$ is in the tangent space $T_x\mathcal{M}$ of $\mathcal{M}$. In other words, $f : \mathcal{M} \to T\mathcal{M}$ defines a smooth vector field on $\mathcal{M}$, as required for Proposition~\ref{thm:ph}. Thus, one can consider the system \eqref{eq:system} in $\mathcal{M}$, or $f$ as representing a smooth vector field on $\mathcal{M}$, and conceptually we are discussing the same thing. 
	
	Note that $\bar x$ is a zero of $f$ if and only if $\bar x$ is a zero of $-f$. In other words, the possibly empty set of zeros of $f$ and $-f$ are the same (at this stage, we have not established the existence of any zero $\bar x \in \mathcal{M}$). Denote $g = -f$ as vector-valued function representing the ``negative'' vector field, i.e. at any $x$, $f(x)$ and $g(x)$ point in the opposite direction.
	
	For any square matrix $A$ the product of its eigenvalues is equal to $\det(A)$. Suppose that $df_{\bar x}$ is Hurwitz for some $\bar x\in\mathcal{M}$. Then, all eigenvalues of $dg_{\bar x} = - df_{\bar x}$ have positive real part, and one has $\det(dg_{\bar x})  > 0$. For any $\bar x\in\mathcal{M}$ satisfying $f(\bar x) = 0$ and $df_{\bar x}$ is Hurwitz, we therefore have $\sign \det(dg_{\bar x}) = +1$, and $dg_{\bar x}$ is orientation preserving.
	
	We are now ready to apply Proposition~\ref{thm:ph} to the vector field $g = -f$ on the manifold $\mathcal{M}$. We know that \textit{if $\bar x$ is a zero of $g$} (and if it exists), then it is nondegenerate by hypothesis and thus $\sign \det(dg_{\bar x}) = \pm 1$.
	Now, the hypothesis that $f$ points inwards at every $x \in \partial\mathcal{M}$ is equivalent to having the vector field $g$ point outwards at every $x\in \partial\mathcal{M}$. Then, \eqref{eq:poincare_hopf_sum} yields
	\begin{align}\label{eq:ph_sum}
	\sum_{i}\sign \det(dg_{\bar x_i}) & = \chi(\mathcal{M}) = 1, 
	\end{align}  
	since $\mathcal{M}$ is contractible. 
	Because $\sign \det(dg_{\bar x_i}) = \pm 1$, there must be at least one zero of $g$ contributing to the left-hand side of \eqref{eq:ph_sum}: we have established the existence of at least one isolated zero $\bar x_1 \in \mathcal{M}$. The hypothesis that $df_{\bar x_i}$ is Hurwitz implies that  $\sign \det(dg_{\bar x_i}) = +1$ for every $\bar{x}_i$, as established in the preceding paragraph. This immediately proves the uniqueness of $\bar{x}_1 = x^*$. Recalling that the set of zeros of $f$ and $g = -f$ are the same establishes the theorem claim. Since $df_{\bar x}$ is Hurwitz, the Linearization Theorem \cite[Theorem 5.41]{sastry1999nonlinearbook} establishes the local exponential stability of $x^*$. Note that the analysis also tells us that $x^* \in \intr(\mathcal{M})$. 
\end{proof}

Since the Poincar\'{e}--Hopf Theorem does not restrict the manifold in consideration to be in the positive orthant of $\mathbb{R}^n$, Theorem~\ref{thm:unique} does not impose that the state of \eqref{eq:system} satisfy $x(t) \geq \vect 0_n$. Three system models in the natural sciences are subsequently considered, which are such that $x(t) \in \mathbb{R}^n_{\geq 0}$. Similarly, other works discussed in the Introduction have applied the Poincar\'{e}--Hopf Theorem to specific systems (as opposed to general systems in Theorem~\ref{thm:unique}), and several systems are not restricted to $\mathbb{R}^n_{\geq 0}$, e.g. \cite{belabbas2013_formationstable,christensen2017PH_potentialgames}. Identifying a suitable manifold $\mathcal{M}$ requires some knowledge of the specific system. For instance, if \eqref{eq:system} has a trivial equilibrium, i.e. at the origin $\vect 0_n$, and one is interested using Theorem~\ref{thm:unique} to study a non-trivial equilibrium, then $\mathcal{M}$ cannot contain the origin. 


\begin{remark}
	Note that the wording chosen in the second to last sentence of the theorem statement is deliberate. For general nonlinear $f$, it may not even be easy to establish the existence of an equilibrium $\bar x \in \mathcal{M}$, let alone whether $\bar x $ is unique. Nonetheless, one does not require knowing the existence or otherwise of $\bar x$ to evaluate $df_{x}$. Then, one can obtain an expression for $df_{\bar x}$ (and perhaps determine whether it is Hurwitz) by leveraging the equality $f(\bar x) = 0$, even if existence of such a $\bar x$ has not been established.
\end{remark}

\begin{remark}
In \cite[Lemma 4.1]{lajmanovich1976SISnetwork}, it is shown by an application of Brouwer's Fixed-Point Theorem \cite{khamsi2011metric_space_book} that if the compact and convex set $\mathcal{M}$ is positively invariant for the system \eqref{eq:system} and $f$ is Lipschitz in $\mathcal{M}$, then there exists at least one equilibrium $\bar x \in \mathcal{M}$. However, unlike Theorem~\ref{thm:unique}, the uniqueness of $\bar x$ or any stability properties cannot be concluded. Moreover, Theorem~\ref{thm:unique} relaxes the requirement that $\mathcal{M}$ be convex, since a great number of contractible manifolds are nonconvex. For example, if there exists a $x_0 \in \mathcal{M}$ such that for all $x \in \mathcal{M}$ and $t\in [0,1]$, the point $tx_0 + (1-t)x \in \mathcal{M}$, then $\mathcal{M}$ is contractible; such an $\mathcal{M}$ is sometimes called a star domain.
\end{remark}

\section{Deterministic Network Models of Epidemics}\label{sec:epidemics}
In this section, as a first illustration, we apply Theorem~\ref{thm:unique} to the familiar deterministic SIS network model introduced in Section~\ref{ssec:SIS_intro}. We require some additional notation and existing linear algebra results. 

A \textit{Metzler} matrix is a matrix which has off-diagonal entries that are all nonnegative \cite{berman1979nonnegative_matrices}. 
A matrix $ A \in \mathbb{R}^{n\times n}$ with all off-diagonal entries nonpositive is called an $M$-matrix if it can be written as $ A = s I_n - B$, with $s > 0$, $ B \geq \mat 0_{n\times n}$ and $s \geq \rho(B)$ \cite{berman1979nonnegative_matrices}. The following results on Metzler matrices and $M$-matrices will prove useful for later analysis\footnote{Such matrices are related to ``compartmental matrices'' studied in models of chemical reaction systems, ecosystems, etc.~\cite{jacquez1993qualitative_compartmental}.}.

\begin{lemma}\label{lem:metzler}
	Let $A$ be an irreducible Metzler matrix. Then, $s(A)$ is a simple eigenvalue of $A$ and there exists a unique (up to scalar multiple) vector $x > \vect 0_n$ such that $Ax = s(A) x$. Let $z \geq \vect 0_n$ be a given non-zero vector. If $Az \leq \lambda z$ for some scalar $\lambda$, then $s(A) \leq \lambda$, with equality if and only if $Az = \lambda z$. If $Az \geq \lambda z$ and $Az \neq \lambda z$, for some scalar $\lambda$, then $s(A) > \lambda$.
\end{lemma} 
The first half of the lemma is a direct consequence of the Perron--Frobenius Theorem for nonnegative matrices \cite[Theorem 2.1.4]{berman1979nonnegative_matrices}. The second part can be obtained from a straightforward application of \cite[Theorem 2.1.11]{berman1979nonnegative_matrices}.

\begin{lemma}[{\cite[Theorem 6.4.6]{berman1979nonnegative_matrices}}]\label{lem:M_matrix_singular}
	Let $R \in \mathbb{R}^{n\times n}$ have all off-diagonal entries nonpositive. Then, the following statements are equivalent
	\begin{enumerate}
		\item $R$ is an $M$-matrix
		\item The eigenvalues of $R$ have nonnegative real parts. 
	\end{enumerate}
\end{lemma}

\begin{lemma}[{\cite[Theorem 4.31]{qu2009cooperative_book}}]\label{lem:M_matrix_singular_irr}
	Suppose that $R$ is a singular irreducible $M$-matrix. If $Q$ is a nonnegative diagonal matrix with at least one positive diagonal element, then the eigenvalues of $R + Q$ have strictly positive real parts. 
\end{lemma}


\subsection{A Unique Endemic Equilibrium for the Network SIS Model}\label{ssec:SIS_ph}

To begin, notice from Proposition~\ref{prop:SIS_network_convergence} that the matrix $-D+B$ uniquely determines the equilibria, and the convergence behaviour of the SIS network system \eqref{eq:SIS_contact_network}. We are interested in applying Theorem~\ref{thm:unique} for $s(-D+B) > 0$ to prove the system \eqref{eq:SIS_contact_network} has a unique endemic equilibrium $x^* \in \intr(\Xi_n)$. Later, we apply the same tool to prove a more powerful result on decentralised control of the SIS model, providing a further generalisation of Proposition~\ref{prop:SIS_network_convergence}. First, we need to find a contractible manifold $\mathcal{M}$ for the system \eqref{eq:SIS_contact_network} with the property that at all points on the boundary $\partial\mathcal{M}$, 
\begin{equation}\label{eq:vf_SIS}
f(x) = (-D+B-XB)x
\end{equation}
is pointing inward. We now identify one such $\mathcal{M}$. 

Since $B$ is nonnegative, $-D+B$ is a \textit{Metzler} matrix. Let  $\phi \triangleq s(-D + B)$, where $y > \vect 0_n$ satisfies $(-D + B)y = \phi y$ in accord with Lemma~\ref{lem:metzler}. Without loss of generality, assume $\max_i y_i = 1$. For a given $\epsilon \in (0,1)$, define the set
\begin{equation}\label{eq:M_definition}
\mathcal{M}_\epsilon \triangleq \{x : \epsilon y_i \leq x_i\leq 1, \forall i = 1, \hdots, n \}.
\end{equation}
The boundary $\partial\mathcal{M}_\epsilon$, is the union of the faces 
\begin{subequations}\label{eq:faceM}
	\begin{align}
	P_i & = \{x :  x_i = \epsilon y_i, x_j \in [\epsilon y_j, 1]\,\forall j \neq i\}, \label{eq:face1M} \\
	Q_i & = \{x :  x_i = 1, x_j \in [\epsilon y_j, 1]\,\forall j \neq i\}. \label{eq:face2M} 
	\end{align}
\end{subequations}
Note that $\mathcal{M}_\epsilon \subset \Xi_n$ for all $\epsilon \in (0,1)$, where $\Xi_n$ is given in \eqref{eq:Xi}. This manifold, and related manifolds, will be used in our application of Theorem~\ref{thm:unique} below. To this end, we state the following lemma, with the proof given in Appendix~\ref{app:lem_pf}.	

\begin{lemma}\label{lem:positive_invariance}
	Consider the system \eqref{eq:SIS_contact_network}, and suppose that $\mathcal{G}=(\mathcal{V}, \mathcal{E}, B)$ is strongly connected. Suppose further that $\phi \triangleq s(-D + B) > 0$. Then, there exists a sufficiently small $\epsilon > 0$ such that $\mathcal{M}_\epsilon$ in \eqref{eq:M_definition} and $\intr(\mathcal{M}_\epsilon)$ are both positive invariant sets of \eqref{eq:SIS_contact_network}, and 
	\begin{subequations}\label{eq:boundary_point}
		\begin{align}
			-\mathbf{e}_i^\top \dot{x} & < 0 \qquad \forall\; x \in P_i\,, i = 1, \hdots n \label{eq:boundary_point1} \\
			\mathbf{e}_i^\top \dot{x} & < 0 \qquad \forall\; x \in Q_i\,, i = 1, \hdots n \label{eq:boundary_point2}
		\end{align}
	\end{subequations}
	where $\mathbf{e}_i$ is the $i$th canonical unit vector. Moreover, if $x(0) \in \partial\Xi_n\setminus \vect 0_n$, then $x(\bar \kappa) \in \mathcal{M}_{\epsilon}$ for some finite $\bar\kappa > 0$ and if $x \in \Xi_n \setminus \vect 0_n$ is an equilibrium of \eqref{eq:SIS_contact_network}, then $x \in \intr(\mathcal{M}_\epsilon)$.
\end{lemma}

We now explain the intuition behind Lemma~\ref{lem:positive_invariance}, and refer the reader to the helpful diagram in Fig.~\ref{fig:manifold_SIS} for an illustrative example. The inequalities \eqref{eq:boundary_point} imply that the vector field represented by $f(x)$ in \eqref{eq:vf_SIS} \textit{points inward} at all points on the boundary $\partial\mathcal{M}_\epsilon$. Notice that $\mathcal{M}_\epsilon$ is an $n$-dimensional hypercube, so it is contractible, but not \textit{smooth}. Specifically, $\mathcal{M}_\epsilon$ is not smooth on the edges and corners formed by the intersection of the faces defined in \eqref{eq:faceM}. 

In order to apply Theorem~\ref{thm:unique}, we shall therefore consider the system \eqref{eq:SIS_contact_network} on a manifold $\tilde{\mathcal{M}}_\epsilon$, which is simply $\mathcal{M}_\epsilon$ as defined in \eqref{eq:M_definition}, but with each edge and corner rounded so that $\tilde{\mathcal{M}}_\epsilon$ is a smooth manifold with boundary $\partial\tilde{\mathcal{M}}_\epsilon$. If the corners and edges are rounded by arbitrarily small amounts, then by continuity, $f(x)$ in \eqref{eq:vf_SIS} will also \textit{point inward} at all points on $\partial\tilde{\mathcal{M}}_\epsilon$. Accordingly, Proposition~\ref{prop:nagumo} implies that $\tilde{\mathcal{M}}_\epsilon$ is a positive invariant set of \eqref{eq:SIS_contact_network}. It should be noted that there are numerous ways to define a suitable $\tilde{\mathcal M}_\epsilon$; it is not unique, and most importantly, such a suitable $\tilde{\mathcal M}_\epsilon$ always exists. 

To aid the reader, an example for $n =2$, corresponding to Fig.~\ref{fig:manifold_SIS}, is now provided. With $x = [x_1, x_2]^\top$, we define $\tilde{\mathcal M}_\epsilon = \bigcup_{i=1}^6 \mathcal{A}_i$, where
\begin{align*}
	\mathcal{A}_1 & = \{x : x_1 \in [\epsilon y_1, 1], x_2 \in [\epsilon y_2+r, 1-r]\} \\
	\mathcal{A}_2 & = \{x : x_1 \in [\epsilon y_1 + r, 1-r], x_2 \in [\epsilon y_2, 1]\} \\
	\mathcal{A}_3 & = \{x : (x_1 - (\epsilon y_1 + r))^2 + (x_2 - (\epsilon y_2 + r))^2 = r^2\} \\
	\mathcal{A}_4 & = \{x : (x_1 - (\epsilon y_1 + r))^2 + (x_2 - (1 - r))^2 = r^2\} \\
	\mathcal{A}_5 & = \{x : (x_1 - (1 - r))^2 + (x_2 - (\epsilon y_2 + r))^2 = r^2\} \\
	\mathcal{A}_6 & = \{x : (x_1 - (1 - r))^2 + (x_2 - (1 - r))^2 = r^2\}
\end{align*}
and $0  < r \ll \epsilon$. By taking $r$ sufficiently small, one can preserve the property that $f$ points inward on the boundaries of $\tilde{\mathcal M}_\epsilon$, and $\tilde{\mathcal M}_\epsilon$ is smooth. Geometrically speaking, $\tilde{\mathcal M}_\epsilon$ is the rectangle $\mathcal{M}_\epsilon$ but with each corner replaced by a sector of a circle of radius $r$, and with the sector angle being 90 degrees, precisely as illustrated in Fig.~\ref{fig:manifold_SIS}.

 We are now in a position to illustrate the application of the Poincar\'e--Hopf Theorem, viz. Theorem~\ref{thm:unique}, to the SIS network model. This is done over two theorems, the first being Theorem~\ref{thm:SIS_ph_unique} immediately below. Note that the statement of Theorem~\ref{thm:SIS_ph_unique} does not present new insights, as the results are already known, see Proposition~\ref{prop:SIS_network_convergence}. (In fact, Theorem~\ref{thm:SIS_ph_unique} only provides a local convergence result). Rather, it is the proof technique of Theorem~\ref{thm:SIS_ph_unique}, utilising Theorem~\ref{thm:unique}, that is of interest, and also  is crucial for subsequent extension that identifies an almost global region of attraction.
\begin{figure}
	\centering
	\def\svgwidth{0.8\linewidth}
	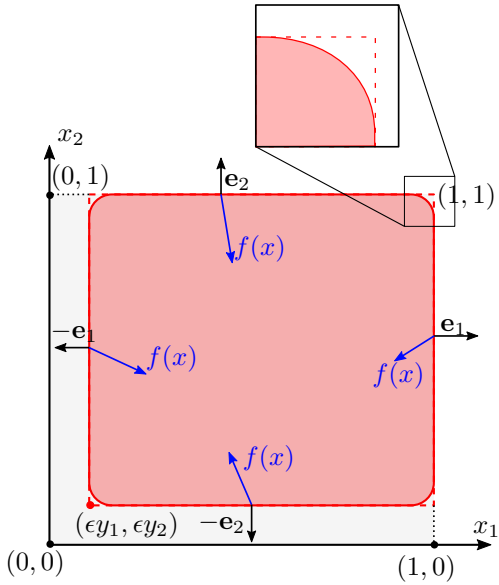
	\caption{An illustration of the compact manifolds $\mathcal{M}_\epsilon$ and $\tilde{\mathcal{M}}_\epsilon$ for \eqref{eq:SIS_contact_network}, with $n = 2$. The cube $\Xi_n$ is in light grey, with dotted black borders, and corners indicated. The dashed red line identifies the boundary of $\mathcal{M}_\epsilon$ (defined in \eqref{eq:M_definition}), and notice the lower corner point of $(\epsilon y_1, \epsilon y_2)$ with exaggerated size (in reality, $\epsilon > 0$ is small). The solid red line identifies $\partial\tilde{\mathcal{M}}_\epsilon$, with the shaded red area being $\intr(\tilde{\mathcal{M}}_\epsilon)$. One can see that $\tilde{\mathcal{M}}_\epsilon$ is simply $\mathcal{M}_\epsilon$ but with the corners rounded so that $\tilde{\mathcal{M}}_\epsilon$ is smooth. The $(1,1)$ corner is magnified to give a clear view. The rounding of corners is exaggerated for clarity; in reality, one only requires an arbitrarily small smoothing of each corner or edge. With reference to \eqref{eq:boundary_point}, black arrows denote canonical unit vectors $\mathbf{e}_i, i = 1, 2$ (with direction), and blue arrows show the vector field $f$ pointing inward at example points on $\partial\tilde{\mathcal{M}}_\epsilon$. }
	\label{fig:manifold_SIS}
\end{figure}

\begin{theorem}\label{thm:SIS_ph_unique}
	Consider the system \eqref{eq:SIS_contact_network}, and suppose that $\mathcal{G} = (\mathcal{V}, \mathcal{E}, B)$ is strongly connected, and $s(-D+B) > 0$. Let $\Xi_n$ be defined in \eqref{eq:Xi}. Then, in addition to the healthy equilibrium $\vect 0_n$, \eqref{eq:SIS_contact_network} has a unique endemic equilibrium $x^*$, satisfying $x^* \in \intr(\Xi_n)$, and $x^*$ is locally exponentially stable.
\end{theorem}
\begin{proof}
Let $\tilde{\mathcal{M}_{\epsilon}}$ be defined as above Theorem~\ref{thm:SIS_ph_unique}, for some sufficiently small $\epsilon > 0$. Lemma~\ref{lem:positive_invariance} implies that any non-zero equilibrium $\bar x$ of \eqref{eq:SIS_contact_network} must satisfy $\bar x \in \intr(\tilde{\mathcal{M}}_{\epsilon})$ and
\begin{equation}\label{eq:endemic_equib}
\vect 0_n = (-D + (I_n - \bar X) B)\bar x.
\end{equation}
This implies that $I_n - \bar X$ is a positive diagonal matrix, and because $B \geq \vect 0_{n\times n}$ is irreducible, $(I_n - \bar X)B \geq \vect 0_{n\times n}$ is also irreducible. Define for convenience $F(x) \triangleq D - (I_n - X)B$. Clearly, $F(x)\,\forall\, x\in \tilde{\mathcal{M}}_{\epsilon}$ has off-diagonal entries that are all nonpositive, and it follows that $-F(\bar x)$ is a \textit{Metzler} matrix for any equilibrium $\bar x \in \tilde{\mathcal{M}}_{\epsilon}$. Lemma~\ref{lem:metzler} and \eqref{eq:endemic_equib} indicate that $s(-F(\bar x)) = 0$, and we conclude using Lemma~\ref{lem:M_matrix_singular} that $F(\bar x)$ is a singular irreducible $M$-matrix.

The Jacobian of $f(\cdot)$ in \eqref{eq:vf_SIS} at $x \in \tilde{\mathcal{M}}_{\epsilon}$ is given by 
\begin{align}
df_x & = - \left(F(x) + \Delta(x) \right)
\end{align}
where $\Delta(x) = \sum_{i=1}^n \left(\sum_{j=1}^n b_{ij} x_j\right) \mathbf{e}_i\mathbf{e}_i^\top$ is a diagonal matrix. Because $B$ is irreducible, there exists for all $i = 1, \hdots, n$, a $k_i$ such that $b_{ik_i} > 0$, which implies that for all $x\in \tilde{\mathcal{M}}_{\epsilon}$ there holds  $ \sum_{j=1}^n b_{ij} x_j \geq b_{ik_i} x_{k_i} > 0$. In other words, $\Delta(x)$ is a positive diagonal matrix for all $x\in \tilde{\mathcal{M}}_{\epsilon}$. It follows immediately from Lemma~\ref{lem:M_matrix_singular_irr} that $F(\bar x) + \Delta(\bar x)$ is a non-singular $M$-matrix, and all of its eigenvalues have strictly positive real parts. In other words, $df_{\bar x}$ is Hurwitz for all $\bar x \in \tilde{\mathcal{M}}_{\epsilon}$ satisfying \eqref{eq:endemic_equib}. Application of Theorem~\ref{thm:unique} establishes that there is in fact a unique equilibrium $x^* \in \intr(\tilde{\mathcal{M}}_{\epsilon})$, and $x^*$ is locally exponentially stable.  
\end{proof}

Existing approaches for proving uniqueness of the endemic equilibrium centre were briefly mentioned in the discussion below Proposition~\ref{prop:SIS_network_convergence}. In the next subsection, we will modify \eqref{eq:SIS_contact_network} via the introduction of decentralised nonlinear feedback controllers. As a consequence, the existing methods of analysis centred around Lyapunov functions and algebraic computations cannot be directly applied, since the system dynamics are significantly changed. On the other hand, we will show that the analysis method of Theorem~\ref{thm:SIS_ph_unique}, which exploits Theorem~\ref{thm:unique}, can be easily extended to include decentralised feedback control, with virtually no change in the analysis complexity. After presenting our results on the controlled SIS network model in Section~\ref{ssec:feedback} below, we will provide a detailed comparison of the framework proposed in this paper, against existing approaches.

\subsection{Decentralised Feedback Control: Challenges and Benefits}\label{ssec:feedback}


Recall that $s(-D+B) > 0$ implies the system \eqref{eq:SIS_contact_network} will converge to the unique endemic equilibrium $x^* \in \intr(\Xi_n)$ as outlined in Proposition~\ref{prop:SIS_network_convergence}. Given the epidemic context, control strategies for the SIS networked system \eqref{eq:SIS_contact_network} almost always have the objective of eliminating the endemic equilibrium by driving the state $x(t)$ to the healthy equilibrium $\vect 0_n$, or at least reducing the infection level at the endemic equilibrium. We give a brief overview of some existing approaches, and refer the reader to \cite{nowzari2016epidemics} for a detailed survey. 

The diagonal entry $d_i > 0$ of the diagonal matrix $D$ represents the recovery rate of the population $i$, while $b_{ij} > 0$ represents the infection rate from population $j \in \mathcal{N}_i$ to population $i$. A common, centralised approach is to formulate and solve an optimisation problem to minimise (and possibly render negative) the value $s(-D +B)$ by setting constant values for parameters $d_i$ or $b_{ij}$, perhaps with certain ``budget'' constraints \cite{preciado2014epidemic_optimal,watkins2016optimal_virus}. The approach can be made partially decentralised~\cite{wan2008control_virus,torres2016sparse_spreading}. A distributed method was recently proposed, but requires a synchronised stopping time across the network and an additional consensus process to compute a piece of centralised information \cite{mai2018_supress_epidemic}.

In contrast, we suppose that we can \textit{dynamically} control (and in particular increase) the recovery rate at node $i$, using a feedback controller. Specifically, we replace $d_i$ in \eqref{eq:SIS_contact_node} with $\bar d_i(t) = d_i + u_i(t)$, where $d_i > 0$ is the constant \textit{base recovery rate}\footnote{We have assumed that $d_i > 0$ to ensure consistency with \eqref{eq:SIS_contact_node}.} intrinsic to population $i$, and $u_i(t)$ is the injected control input at node $i$. We first give some assumptions on $u_i(t)$, before providing motivation and explanation.

In this paper, we consider the general class of decentralised, local state feedback controllers of the form
\begin{equation}\label{eq:gen_controller}
u_i(t) = h_i(x_i(t)),
\end{equation}
where $h_i : [0,1] \to \mathbb{R}_{\geq 0}$ is bounded, smooth and monotonically nondecreasing, satisfying \mbox{$h_i(0) = 0$}. We are motivated to consider \eqref{eq:gen_controller} for practical reasons. The control effort $u_i(t) = h_i(x_i(t))$ may represent pharmaceutical interventions, drug medication, or additional hospital resources, which allow infected individuals to more rapidly recovery from the disease. For instance, zinc supplements have been reported to decrease the period of infection for the common cold~\cite{hemila2017zinc}. Assuming that $h_i \geq 0$ is nondecreasing in $x_i$ yields an intuitive feedback control strategy: additional resources are introduced into node $i$ to increase (or keep constant) the recovery rate $\bar d_i(t)$ as the infection proportion $x_i(t)$ increases. For population~$i$, \eqref{eq:gen_controller} only requires the local state information $x_i(t)$, which has the advantage of decentralised implementation. This contrasts with many existing approaches, such as those described above which require centralised design or implementation, including information regarding $D$ and $B$. The work \cite{liu2019bivirus} considers $\bar d_i$ of the special form $h_i(x_i) = k_i x_i$ with $k_i >0$ and $d_i = 0$. 

The network dynamics become
\begin{equation}\label{eq:network_SIS_contact_control}
\dot{x}(t) = (- D - H(x(t)) + B - X(t)B)x(t), 
\end{equation}
where $H(x(t)) = \diag(h_1(x_1(t)), \hdots, h_n(x_n(t)))$ is a nonnegative diagonal matrix. Notice the right side of the equation is in general no longer quadratic in $x$. It is straightforward to verify that if $x(0)\in \Xi_n$, then $x(t) \in \Xi_n$ for all $t\geq 0$. In accordance with intuition, the following establishes that when $s(- D + B) \leq 0$, the controlled network system \eqref{eq:network_SIS_contact_control} retains the convergence properties of the uncontrolled system \eqref{eq:SIS_contact_network} as noted earlier in Proposition~\ref{prop:SIS_network_convergence}.
\begin{theorem}\label{thm:healthy}
	Consider the system \eqref{eq:network_SIS_contact_control}, with $\mathcal{G} = (\mathcal{V}, \mathcal{E}, B)$ strongly connected, and $\Xi_n$ defined in \eqref{eq:Xi}. Suppose that $s(-D+B) \leq 0$ and for all $i\in \mathcal{V}$, $h_i : [0,1] \to \mathbb{R}_{\geq 0}$ is bounded, smooth and monotonically nondecreasing, satisfying \mbox{$h_i(0) = 0$}. Then, $\vect 0_n$ is the unique equilibrium of \eqref{eq:network_SIS_contact_control} in $\Xi_n$, and $\lim_{t\to\infty} x(t) = \vect 0_n$ for all $x(0) \in \Xi_n$.
\end{theorem}
\begin{proof}
	Suppose that $x^*$ is a nonzero equilibrium of \eqref{eq:network_SIS_contact_control}. A simple adjustment to Lemma~\ref{lem:positive_invariance} yields that $\vect 0_n < x^* < \vect 1_n$. If $s(-D+B) \leq 0$, then according to Lemma~\ref{lem:M_matrix_singular}, $D - B$ is an irreducible $M$-matrix. Since $I_n - X^*$ is a strictly positive diagonal matrix, $s((I_n - X^*)B) < s(B)$ according to \cite[Corollary 2.1.5]{berman1979nonnegative_matrices}. Combining this with the fact that $H(x^*)$ is nonnegative diagonal, we can use Lemma~\ref{lem:M_matrix_singular_irr} and the definition of an $M$-matrix at the start of Section~\ref{sec:epidemics} to conclude that $D + H(x^*) -(I_n - X^*)B$ is an irreducible nonsingular $M$-matrix. However, the nonsingularity property contradicts the assumption that $x^* > \vect 0_n$ satisfies $(D + H(x^*) -(I_n - X^*)B)x^* = \vect 0_n$ according to \eqref{eq:network_SIS_contact_control}. Thus, there are no endemic equilibria when $s(-D+B) \leq 0$.
	
	From \eqref{eq:network_SIS_contact_control}, we obtain that $\dot{x} \leq \dot{y} = (-D +B)y$ because $I_n - X(t)$ is a diagonal matrix with diagonal entries in $[0, 1]$, and $H(x(t))$ is nonpositive. If $s(-D+B) < 0$, then $-D + B$ is Hurwitz, and initialising $\dot{y} = (-D +B)y$ with $y(0) = x(0)$  yields $\lim_{t\to\infty} x(t) = \vect 0_n$. Convergence when $s(-D +B) = 0$ can be similarly argued.
\end{proof}

The following theorem identifies the outcome of using \eqref{eq:gen_controller} to control \eqref{eq:network_SIS_contact_control} when $s(-D + B) > 0$.

\begin{theorem}\label{thm:impossibility}
	Consider the system \eqref{eq:network_SIS_contact_control}, with $\mathcal{G} = (\mathcal{V}, \mathcal{E}, B)$ strongly connected, and $\Xi_n$ defined in \eqref{eq:Xi}. Suppose that $s(-D + B) > 0$, and that for all $i\in \mathcal{V}$, $h_i : [0,1] \to \mathbb{R}_{\geq 0}$ is bounded, smooth and monotonically nondecreasing, satisfying \mbox{$h_i(0) = 0$}. Then, 
	\begin{enumerate}
		\item In $\Xi_n$, \eqref{eq:network_SIS_contact_control} has two equilibria: $x = \vect 0_n$, and a unique endemic equilibrium $x^* \in \intr(\Xi_n)$, which is unstable and locally exponentially stable, respectively.
		\item For all $x(0) \in \Xi_n \setminus \vect 0_n$, there holds $\lim_{t\to\infty} x(t) = x^*$ exponentially fast.
	\end{enumerate}
\end{theorem}
\begin{remark}
	Theorem~\ref{thm:impossibility} establishes two key properties of the SIS model under feedback control. Item 1 indicates that a stable healthy state cannot be achieved, and a unique endemic equilibrium $x^*$ that is locally exponentially stable continues to exist; it is impossible for the decentralised feedback control to globally stabilise the system to the healthy equilibrium. Item 2 establishes a large region of attraction of the endemic equilibrium. In the sequel, we will show that  feedback control ``improves'' the limiting behaviour: the controlled system converges to an endemic equilibrium which is closer to the origin than the endemic equilibrium of the uncontrolled system.
\end{remark}
\begin{proof}
	The proof consists of two parts. In \textit{Part 1}, we establish the existence and uniqueness of the endemic equilibrium $x^* \in \intr(\Xi_n)$, and the local stability properties of $x^*$ and $\vect 0_n$. In \textit{Part 2}, we establish the convergence to $x^*$.
	
	\textit{Part 1:} Under the theorem hypothesis, $H(x(t))$ is a nonnegative diagonal matrix. It can be shown that if $s(-D + B) > 0$, then Lemma~\ref{lem:positive_invariance} continues to hold when replacing \eqref{eq:SIS_contact_network} with \eqref{eq:network_SIS_contact_control}. Only simply adjustments to the proof of Lemma~\ref{lem:positive_invariance} are needed, which we omit for brevity. To summarise, there exists a sufficiently small $\epsilon > 0$ such that $\mathcal{M}_{\epsilon}$ in \eqref{eq:M_definition} and  $\intr(\mathcal{M}_{\epsilon})$ are both positive invariant sets of \eqref{eq:network_SIS_contact_control}, and for every $x \in \partial\mathcal{M}_{\epsilon}$,
	\begin{equation}\label{eq:vf_SIS_control}
	f(x) = (- D - H(x) + B - X B) x
	\end{equation}
	points inward to $\mathcal{M}$. Similar to the discussion above Theorem~\ref{thm:SIS_ph_unique}, we can obtain from $\mathcal{M}_{\epsilon}$ a smooth and compact manifold $\tilde{\mathcal{M}}_{\epsilon}$, with the property that $f(x)$ in \eqref{eq:vf_SIS_control} also points inward for every $x\in \partial\tilde{\mathcal{M}}_{\epsilon}$. Thus, both $\tilde{\mathcal{M}}_{\epsilon}$ and $\intr(\tilde{\mathcal{M}}_{\epsilon})$ are positive invariant sets of \eqref{eq:network_SIS_contact_control}.  Moreover, there exists a finite $\kappa$ such that for all $x(0) \in \partial \Xi_n\setminus \vect 0_n$, there holds $x(\kappa) \in \tilde{\mathcal{M}_{\epsilon}}$. This implies that any nonzero equilibrium of \eqref{eq:network_SIS_contact_control} must be in $\intr(\tilde{\mathcal{M}}_{\epsilon}) \subset \intr(\Xi_n)$.
	
	Now, suppose that $\tilde x \in \intr(\tilde{\mathcal{M}}_{\epsilon})$ is an equilibrium of \eqref{eq:network_SIS_contact_control}. Then, $\tilde x$ must satisfy $\vect 0_n < \tilde x < \vect 1_n$ and
	\begin{equation}\label{eq:endemic_equib_control}
	\vect 0_n = (- D - H(\tilde x) + (I_n - \tilde X) B)\tilde x.
	\end{equation}
	This implies that $I_n - \tilde X$ is a positive diagonal matrix, and because $B \geq \vect 0_{n\times n}$ is irreducible, $(I_n - \tilde X)B$ is also an irreducible nonnegative matrix. Let us define for convenience $F(x) \triangleq  D + H(x) - (I_n - X)B$. Obviously, $F(x)\,\forall\,x\in \tilde{\mathcal{M}}_{\epsilon}$ has off-diagonal entries that are all nonpositive, and it follows that $-F(\tilde x)$ is a \textit{Metzler} matrix for any equilibrium $\tilde x \in \tilde{\mathcal{M}}_{\epsilon}$. Lemma~\ref{lem:metzler} and \eqref{eq:endemic_equib_control} indicate that $s(-F(\tilde x)) = 0$, and as a consequence, we can use Lemma~\ref{lem:M_matrix_singular} to conclude that $F(\tilde x)$ is a singular irreducible $M$-matrix.
	
	Define 
	\begin{equation}
	\Gamma(x) = \diag(\frac{\partial h_1}{\partial x_1} x_1, \hdots, \frac{\partial h_n}{\partial x_n} x_n),
	\end{equation} 
	and because $h_i$ is monotonically nondecreasing in $x_i$, $\Gamma(x)$ is a nonnegative diagonal matrix for all $x \in \tilde{\mathcal{M}}_{\epsilon}$. The Jacobian of \eqref{eq:network_SIS_contact_control} at a point $x\in\tilde{\mathcal{M}}_{\epsilon}$ is given by 
	\begin{align}
	df_x & = - D - H(x) + B - XB - \Delta(x) - \Gamma(x) \nonumber \\
	& = - \left(F(x) + \Delta(x) +  \Gamma(x) \right) \label{eq:jacobian_control}
	\end{align}
	where $\Delta(x) = \sum_{i=1}^n \big(\sum_{j=1}^n b_{ij} x_j \big) \mathbf{e}_i\mathbf{e}_i^\top$ is a diagonal matrix. Because $B$ is irreducible, there exists for all $i = 1, \hdots, n$, a $k_i$ such that $b_{ik_i} > 0$. This implies that for all $x\in \tilde{\mathcal{M}}_{\epsilon}$ there holds  $\sum_{j=1}^n b_{ij} x_j \geq b_{ik_i} x_{k_i} > 0$. It follows that $\Delta(x)$ is a positive diagonal matrix for all $x\in \tilde{\mathcal{M}}_{\epsilon}$. Lemma~\ref{lem:M_matrix_singular_irr} establishes that $F(\tilde x) + \Delta(\tilde x) + \Gamma(\tilde x)$ is a nonsingular $M$-matrix, with eigenvalues having strictly positive real parts. This implies that $df_{\tilde x}$ is Hurwitz for all $\tilde x \in \tilde{\mathcal{M}}_{\epsilon}$ satisfying \eqref{eq:endemic_equib_control}. Application of Theorem~\ref{thm:unique} establishes that there is in fact a unique equilibrium $x^* \in \intr(\tilde{\mathcal{M}}_{\epsilon}) \subset \intr(\Xi_n)$, and $x^*$ is locally exponentially stable. 
	
	Consider now the healthy equilibrium $\vect 0_n$. Notice that $df_{\vect 0_n} = -F(\vect 0_n) = - D + B$. Since $s(- D + B) > 0$ by hypothesis, the Linearization Theorem \cite[Theorem 5.42]{sastry1999nonlinearbook} yields that $\vect 0_n$ is an unstable equilibrium of \eqref{eq:network_SIS_contact_control}.
	
	\textit{Part 2:} 
	We established above that there exists a finite $\kappa$ such that $x(\kappa) \in \tilde{\mathcal{M}_{\epsilon}}$ for all $x(0) \in \partial \Xi_n\setminus \vect 0_n$. To complete the proof, we only need to show that $\lim_{t\to\infty} x(t) = x^*$ for all $x(0) \in \intr(\tilde{\mathcal{M}}_{\epsilon})$. We shall use key results from the theory of monotone dynamical systems, the details of which are presented Appendix~\ref{app:monotone}.
	
	First, notice that $df_x$ in \eqref{eq:jacobian_control} is an irreducible matrix with all off-diagonal entries nonnegative for all $x \in \intr(\tilde{\mathcal{M}}_{\epsilon})$. Thus, \eqref{eq:network_SIS_contact_control} is a $\mathbb{R}^n_{\geq 0}$ monotone system in $\intr(\tilde{\mathcal{M}}_{\epsilon})$ (see Lemma~\ref{lem:monotone} in Appendix~\ref{app:monotone}, and use $P_m = I_n$). Since $x^*$ is the unique equilibrium of \eqref{eq:network_SIS_contact_control} in the open, bounded and positive invariant set $\intr(\tilde{\mathcal{M}}_{\epsilon}) \subset \mathbb{R}^n_{\geq 0}$, Proposition~\ref{prop:conv} in Appendix~\ref{app:monotone} yields $\lim_{t\to\infty} x(t) = x^*$ asymptotically\footnote{As detailed in Appendix~\ref{app:monotone}, Proposition~\ref{prop:conv} is an extension of a well known result, viz. Lemma~\ref{lem:monotone_converge}, when there is a unique equilibrium.} for all $x(0) \in \tilde{\mathcal{M}}_{\epsilon}$. It remains to prove the convergence is exponentially fast.

	
	Since $df_{x^*}$ is Hurwitz, let $\mathcal{B}$ denote the locally exponentially stable region of attraction of $x^*$. For every $x_0 \in \tilde{\mathcal{M}}_{\epsilon}$, the fact that $\lim_{t\to\infty} x(t) = x^*$ implies that there exists a finite $T_{x_0} \geq 0$ such that $x(0) = x_0$ for \eqref{eq:network_SIS_contact_control} yields $x(t) \in \mathcal{B}$ for all $t\geq T_{x_0}$. Now, $\tilde{\mathcal{M}}_{\epsilon}$ is compact, which implies that there exists a $\bar T \geq \max_{x_0 \in \tilde{\mathcal{M}}_{\epsilon}} T_{x_0}$ such that for all $x(0) \in \tilde{\mathcal{M}}_{\epsilon}$, there holds $x(t) \in \mathcal{B}$ for all $t \geq \bar T$. In other words, there exists a time $\bar T$ independent of $x(0)$, such that any trajectory of \eqref{eq:network_SIS_contact_control} beginning in $\tilde{\mathcal{M}}_{\epsilon}$ enters the region of attraction $\mathcal{B}$ of the locally exponentially stable equilibrium $x^*$. Because $\bar T$ is independent of the initial conditions, there exist positive constants $\alpha_1$ and $\alpha_2$ such that 
		\begin{equation*}
		\Vert x(t) - x^* \Vert \leq \alpha_1 e^{-\alpha_2 t} \Vert x(0) - x^* \Vert
		\end{equation*} 
	for all $x(0) \in \tilde{\mathcal{M}}_{\epsilon}$ and $t\geq 0$. I.e., $\lim_{t\to\infty} x(t) = x^*$ exponentially fast for all $x(0) \in \tilde{\mathcal{M}}_{\epsilon}$.	
\end{proof}

We conclude our analysis of the controlled SIS network model by establishing that decentralised feedback control always pushes the endemic equilibrium closer to the healthy equilibrium (the proof is given in Appendix~\ref{app:lem_pf_eqb}):
\begin{lemma}\label{lem:new_eqb}
	Consider the system \eqref{eq:network_SIS_contact_control}, with $\mathcal{G} = (\mathcal{V}, \mathcal{E}, B)$ strongly connected, and $\Xi_n$ defined in \eqref{eq:Xi}. Suppose that $s(-D + B) > 0$, and that for all $i\in \mathcal{V}$, $h_i : [0,1] \to \mathbb{R}_{\geq 0}$ is bounded, smooth and monotonically nondecreasing, satisfying \mbox{$h_i(0) = 0$} and $\exists j: x_j > 0 \Rightarrow h_j(x_j) > 0$. Let $x^*$ and $\bar x^*$ denote the unique endemic equilibrium of \eqref{eq:SIS_contact_network} and \eqref{eq:network_SIS_contact_control}, respectively. Then, $\bar x^* < x^*$.
\end{lemma}
It is worth noting that the presence of a single node $j$ with positive control, i.e. $x_j > 0 \Rightarrow h_j(x_j) > 0$, leads to an improvement for \textit{every} node $i$, i.e., $\bar x_i^* < x_i^*$. This would not be expected if $\mathcal{G}$ was not strongly connected.

\subsection{An Illustrative Simulation Example}
In this subsection, we provide a simple simulation example of a controlled SIS system \eqref{eq:network_SIS_contact_control} with $n = 2$ nodes. The aim is to illustrate the impact on the SIS network dynamics via the introduction of feedback control, to provide an intuitive explanation and discuss the implications of Theorem~\ref{thm:impossibility} and Lemma~\ref{lem:new_eqb}. We therefore choose the parameters and specific control functions $h_i$ arbitrarily; the salient conclusions presented below are unchanged for many other choices of parameters and controllers. We set 
\begin{equation*}
	D = \begin{bmatrix}0.3 & 0 \\ 0 & 0.8\end{bmatrix},\quad B = \begin{bmatrix}0.2 & 0.5 \\ 0.7 & 0.1\end{bmatrix},
\end{equation*}
which yields $s(- D + B) = 0.2633$. 

When there is no control, i.e. $h_1(x_1) \equiv h_2(x_2) \equiv 0$, the vector-valued function 
\begin{equation}\label{eq:vf_SIS_original}
	f(x) = (-D + (I_n-X)B)x,
\end{equation}
defines the dynamics of \eqref{eq:SIS_contact_network}, and represents the vector field shown in Fig.~\ref{fig:SIS_nocontrol}. Since $s(- D + B) > 0$, the endemic equilibrium $x^* = [0.4413, 0.2973]^\top$ (the red dot) is attractive for all $x(0)\in \Xi_n \setminus \vect 0_n$, as per Proposition~\ref{prop:SIS_network_convergence}, Item 2, and Theorem~\ref{thm:impossibility}. 

We then introduce the feedback controllers $h_1(x_1) = 0.5{x_1}^{0.5}$ and $h_2(x_2) = 0.9x_2$ into the SIS network model, given in \eqref{eq:network_SIS_contact_control}. The resulting vector field is shown Fig.~\ref{fig:SIS_control}, represented by the vector-valued function
\begin{equation}\label{eq:vf_SIS_new}
	\bar f(x) = (-D - H(x) + (I_n-X)B)x.
\end{equation}
Although the introduction of the $h_i$ has modified the vector-valued function to become $\bar f$, there remains a unique zero in $\intr(\Xi_n)$. In the context of the SIS model, there is a unique endemic equilibrium $\bar x^* = [0.15, 0.1142]^\top$, and consistent with Theorem~\ref{thm:impossibility}, all trajectories with $x(0)\in \Xi_n \setminus \vect 0_n$ converge to $\bar x^*$.
Comparing Fig.~\ref{fig:SIS_nocontrol} and \ref{fig:SIS_control}, one sees that the feedback control has shifted the endemic equilibrium from $x^*$ to $\bar x^*$, which clearly obeys the inequality $\bar x^* < x^*$ as detailed in Lemma~\ref{lem:new_eqb}. 

\begin{figure}
	\centering
	\includegraphics[width=0.65\linewidth]{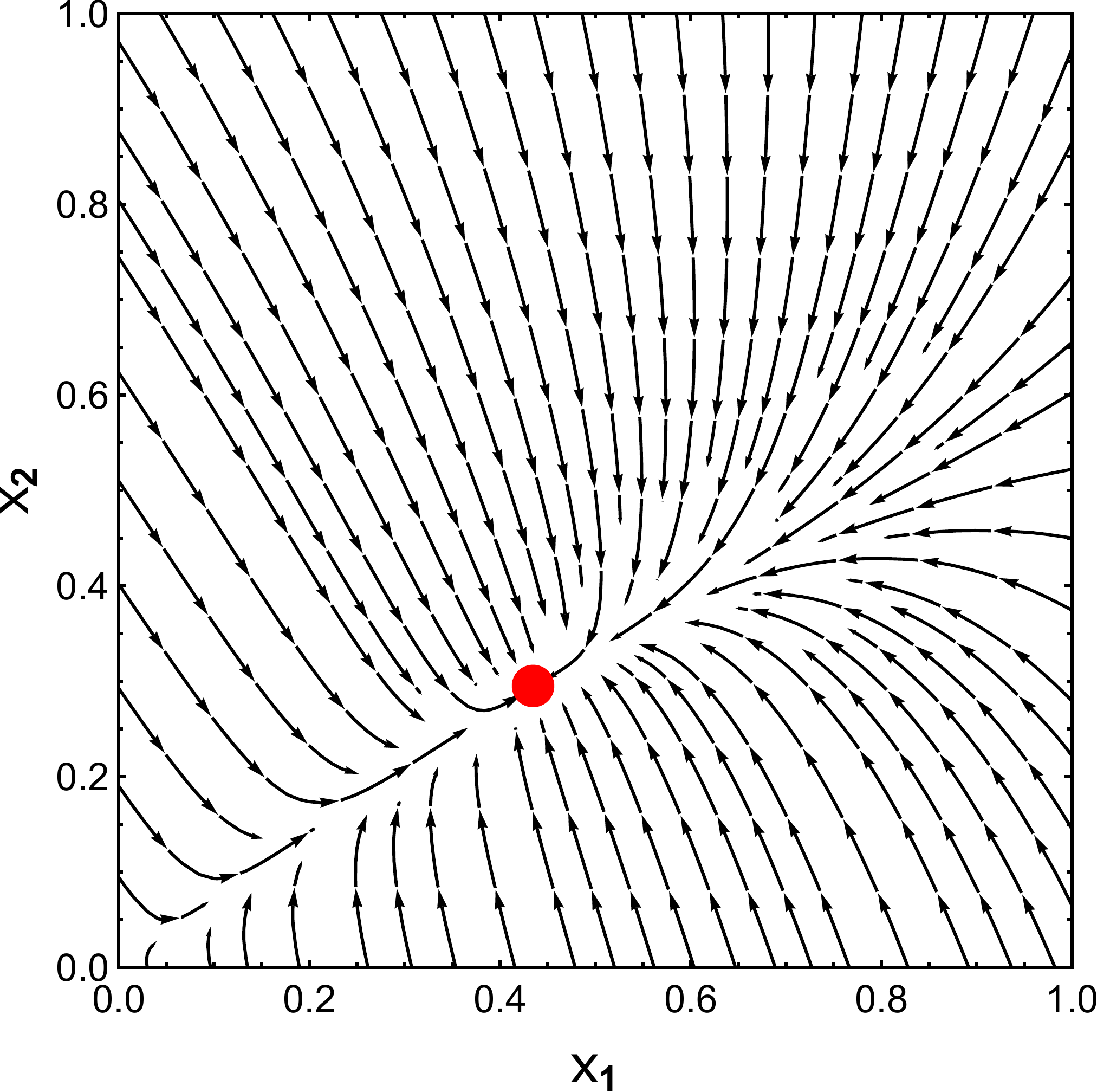}
	\caption{Vector field of an uncontrolled SIS network model with $2$ nodes. The red dot identifies the unique endemic equilibrium $x^* =[0.4413, 0.2973]^\top$. }
	\label{fig:SIS_nocontrol}
\end{figure}

\begin{figure}
	\centering
	\includegraphics[width=0.65\linewidth]{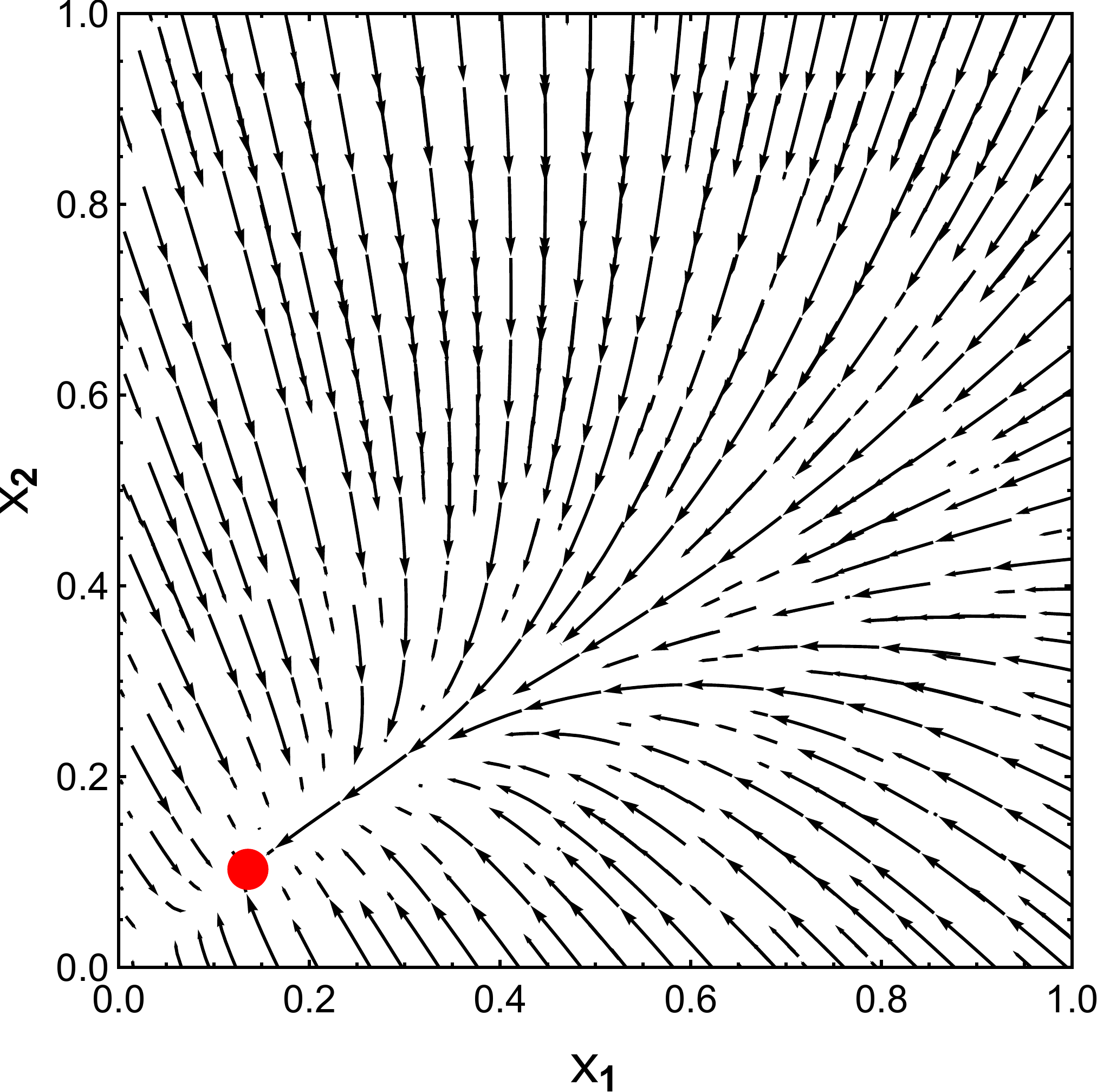}
	\caption{Vector field of a controlled SIS network model with $2$ nodes. The red dot identifies the unique endemic equilibrium $\bar x^* = [0.15, 0.1142]^\top$. Although the feedback control $h_i(x_i(t))$ shifts $\bar x^*$ closer to the origin (the healthy equilibrium) compared to $x^* = [0.4413, 0.2973]^\top$ of the uncontrolled network (see Fig.~\ref{fig:SIS_nocontrol}), all trajectories of the controlled SIS network converge to $\bar x^*$ except $x(0) = \vect 0_2$.}
	\label{fig:SIS_control}
\end{figure}

%


To summarise, Theorem~\ref{thm:impossibility} provides us with conclusions on a broad class of decentralised feedback controllers. Specifically, if the underlying uncontrolled system has a unique endemic equilibrium (that is convergent for all $x(0)\in \Xi_n \setminus \vect 0_n$), then no matter how we design the control functions $h_i(x_i)$, there will always be a unique endemic equilibrium that is convergent for all $x(0)\in \Xi_n \setminus \vect 0_n$. In the language of Theorem~\ref{thm:unique}, introduction of $h_i$ modifies the vector-valued function $f$ in \eqref{eq:vf_SIS_original} to become $\bar f$ in \eqref{eq:vf_SIS_new}, but does not change the stability of the Jacobian at any zero in $\mathcal{M}$, preserving the uniqueness property. This demonstrates the challenge of attempting decentralised control where each node only utilises local state $x_i$ for information. Nonetheless, Lemma~\ref{lem:new_eqb} demonstrates that some health benefits are obtained, since there is a smaller fraction of infected individuals at each population.

From Theorem~\ref{thm:healthy}, the healthy equilibrium is globally asymptotically stable for the controlled network if and only if the \textit{underlying uncontrolled network} itself has the property that $x = \vect 0_n$ is globally asymptotically stable, i.e. $s(-D + B) \leq 0$. 
If $s(- D + B) > 0$, one may wish to consider other decentralised or distributed control methods, including controlling the infection rates as functions of $x$, e.g. $b_{ij}(x_j(t), x_i(t))$. If only local information $x_i(t)$ for population $i$ is available, one might require nonsmooth or time-varying or adaptive controllers; in this case the Poincar\'e-Hopf theorem, and consequently Theorem~\ref{thm:unique}, is not applicable (at least not without significant modifications). 

\subsection{Discussions on the Analysis Framework}
To conclude this section, we use the result of Theorem~\ref{thm:impossibility} to drive our discussion on the strengths and weaknesses of our analysis framework, including especially Theorem~\ref{thm:unique}.
	
Theorem~\ref{thm:unique} establishes local exponential stability of the unique equilibrium $x^*$, and notice that this does not exclude the possibility of limit cycles or chaos. Indeed, in the next section, we consider a different biological system model known to exhibit limit cycles and chaotic behaviour, and derive a sufficient condition for a unique equilibrium $x^*$ without requiring global convergence to $x^*$. However, if one believes there is global convergence, then \textit{proving} the uniqueness of $x^*$ can be a critical step, as it may inform of potential approaches for studying convergence.
In existing work, a sufficient condition is identified using monotone systems theory that guarantees  \eqref{eq:general_system} for all initial conditions either converges to $\vect 0_n$, or converges to a unique equilibrium $x^*$ in $\mathbb{R}^n_{> 0}$~\cite[Theorem 11]{jacquez1993qualitative_compartmental}. However, such a result does not help to establish whether  $x^*$ actually exists. Moreover, \cite[Theorem 11]{jacquez1993qualitative_compartmental} cannot be used to obtain Theorem~\ref{thm:impossibility}, since the required condition is satisfied for the uncontrolled SIS system \eqref{eq:SIS_contact_network}, but is not guaranteed to be satisfied for the controlled SIS system \eqref{eq:network_SIS_contact_control}.

One advantage of Theorem~\ref{thm:unique} is that it enables one to analyse modifications to existing models, without significantly changing the analysis technique. In Section~\ref{ssec:feedback}, we modified the standard SIS network model by adding decentralised feedback control, in the process destroying the quadratic character of the differential equation set. Application of Theorem~\ref{thm:unique} to cover a broad class of controllers involves only minor adjustments between the analysis of the uncontrolled network (Theorem~\ref{thm:SIS_ph_unique}) and the controlled network (Theorem~\ref{thm:impossibility}). The same technique as in Theorem~\ref{thm:impossibility} can also be used to analyse another variant of the SIS model, known as the SIS network model in a ``patchy environment''~\cite{wang2004epidemic_patchy,jin2005effect_patchy}, which aims to capture infection due to individuals travelling between population nodes. 

In contrast, the algebraic-based approach of \cite{khanafer2016SIS_positivesystems,mei2017epidemics_review,fall2007SIS_model} has the advantage of being able to explicitly compute (albeit in a centralised way) the endemic equilibrium for \eqref{eq:SIS_contact_network}, and examine weakly connected networks \cite{khanafer2016SIS_positivesystems} or nodes with zero recovery rates, $d_i = 0$ \cite{liu2019bivirus}. However, such an approach cannot be easily adapted to the patchy environment SIS model or handle broad classes of feedback control as in Section~\ref{ssec:feedback}. The Lyapunov-based approach of \cite{shuai2013epidemic_lyapunov} can be used to study the patchy environment models \cite{li2009patchy_SIR}, and has the advantage of simultaneously establishing the uniqueness of the endemic equilibrium $x^*$ and global convergence. However, the Lyapunov functions of \cite{shuai2013epidemic_lyapunov,li2009patchy_SIR} cannot be extended to consider the model with decentralised feedback control, as we did in Section~\ref{ssec:feedback}. One requires searching for a new Lyapunov function, which may be difficult to identify and also may not cover a broad class of controllers as we have.

One drawback of Theorem~\ref{thm:unique} is the need to identify a manifold $\mathcal{M}$ which has specific properties, such as being contractible and with $f$ pointing inward to $\mathcal{M}$ at every point on the boundary $\partial\mathcal{M}$. This is not always straightforward. Nonetheless, many models in the natural sciences will be well-defined for some positively invariant set of the state space, which may be used as a basis to inform the identification of $\mathcal{M}$. Once $\mathcal{M}$ has been identified, it may be reused for modified variants of the system. In Section~\ref{ssec:SIS_ph}, we identify an appropriate $\mathcal{M}$ for the uncontrolled SIS network model given in \eqref{eq:SIS_contact_network}; a near identical $\mathcal{M}$ can be used for the controlled SIS network model given in \eqref{eq:network_SIS_contact_control}.
 
Last, and as we will next demonstrate, Theorem~\ref{thm:unique} can be applied to other classes of models; one can view Theorem~\ref{thm:unique} as a useful tool in the first step in analysing general systems given by \eqref{eq:general_system}, and not as a stand-alone, all powerful result.  

\section{Lotka--Volterra Systems}\label{sec:other}
We now illustrate the versatility of Theorem~\ref{thm:unique} (from Section~\ref{sec:PH_general}), by applying it to a different model in the natural sciences. While Section~\ref{sec:epidemics} considers the SIS networked epidemic model, including with feedback control, this section considers the generalised nonlinear Lotka--Volterra model. We use Theorem~\ref{thm:unique} to identify conditions for the existence of a unique non-trivial equilibrium, and establish its local stability property. We require some additional linear algebra results. For a matrix $A\in\mathbb R^{n\times n}$, whose diagonal entries satisfy $a_{ii} > 0,\forall\,i = 1, \hdots, n$, consider the following four conditions:
\begin{enumerate}[label=C\arabic*]
	\item \label{C1}
	There exists a positive diagonal $D$ such that $AD$ is strictly diagonally dominant, i.e. there holds
	\begin{equation}\label{eq:ineq_diag_dom}
	d_ia_{ii}>\sum_{j\neq i}^n d_j\vert a_{ij}\vert  \,,\;\forall\, i=1,2,\dots, n.
	\end{equation}
	\item \label{C2}
	There exists a diagonal positive $C$ for which $AC+CA^{\top}$ is positive definite. 
	\item \label{C3}
	All the leading principal minors of $A$ are positive.
	\item \label{C4}
	$A$ is a nonsingular $M$-matrix (see Lemma~\ref{lem:M_matrix_singular}).
\end{enumerate}
Two simple lemmas flow from this.

\begin{lemma}[\hspace{-0.5pt}{\cite[Chapter 6]{berman1979nonnegative_matrices}}]\label{lem:matrix_conditions}
	Consider a matrix $A\in\mathbb R^{n\times n}$, with diagonal entries satisfying $a_{ii} > 0,\forall\,i = 1, \hdots, n$. Then, \ref{C1} $\Rightarrow$ \ref{C2} $\Rightarrow$ \ref{C3}. If further $A$ has all off-diagonal elements nonpositive, then \ref{C1} $\Leftrightarrow$ \ref{C2}  $\Leftrightarrow$ \ref{C3}  $\Leftrightarrow$ \ref{C4}.
\end{lemma}

\begin{lemma}\label{lem:1}
	Let $A \in\mathbb R^{n\times n}$ be a matrix such that \ref{C1} holds, with diagonal entries satisfying $a_{ii} > 0,\forall\,i = 1, \hdots, n$. Suppose $B\in\mathbb R^{n\times n}$ is a matrix related to $A$ by
	\begin{eqnarray*}
	b_{ii}&\geq&a_{ii}\\\nonumber
	|b_{ij}|&\leq&|a_{ij}|\,,\;i\neq j
	\end{eqnarray*}
	Then $B$ satisfies \ref{C1} with the same matrix $D$ as used in the defining strict diagonal dominance inequalities for $A$.
\end{lemma}
The proof is straightforward, and follows by observing that each of the inequalities in \eqref{eq:ineq_diag_dom} holds with $a_{ii}$ and $a_{ij}$ replaced by $b_{ii}$ and $b_{ij}$ respectively, in view of the inequalities in the lemma statement.

\subsection{Generalised Nonlinear Lotka--Volterra Models}\label{ssec:lotka_volterra}

The basic Lotka--Volterra models consider a population of $n$ biological species, with the variable $x_i$ associated to species $i \in \{1, \hdots, n\}$. Typically, $x_i \geq 0$ denotes the population size of species $i$ and has dynamics
\begin{equation}\label{eq:genLVxi}
\dot x_i(t)=d_ix_i(t)+\Big(\sum_{j=1}^n a_{ij}x_j(t)\Big)x_i(t)
\end{equation}
With $x = [x_1, \hdots, x_n]^\top$, the matrix form is given by
\begin{equation}\label{eq:genLV}
\dot x(t)=(D+X(t)A)x(t)
\end{equation}
where $D = \diag(d_1, \hdots, d_n)$ and $X= \diag (x_1, \hdots, x_n)$. For the matrix $A$, there are no a priori restrictions on the signs of the $a_{ij}$, though generally diagonal terms are taken as negative. It is understood that almost exclusively, interest is restricted to systems in \eqref{eq:genLV} with $D$ and $A$ such that $x(0) \in \mathbb{R}^n_{\geq 0}$ implies $x(t) \in \mathbb{R}^n_{\geq 0}$ for all $t\geq 0$. That is, the positive orthant $\mathbb{R}^n_{\geq 0}$ is a positive invariant set of \eqref{eq:genLV}. 

The literature on Lotka-Volterra systems is vast. We note a small number of key aspects. For an introduction, one can consult \cite{takeuchi1996global}. Many behaviours can be exhibited; indeed, \cite{smale1976differential} establishes that an $n$-dimensional Lotka-Volterra system can be constructed with the property that trajectories converge to an $(n-1)$-dimensional linear subspace in which the motion can follow that of \textit{any} $(n-1)$-dimensional system. Since a second-order system can exhibit limit cycles and a third-order system can exhibit limit cycles, strange attractors or chaos, these behaviours can be found in third or fourth (or higher) order Lotka-Volterra systems. The original prey-predator system associated with the names Lotka and Volterra is second order, and can display nonattracting limit cycles, as well as having a saddle point equilibrium and a nonhyperbolic equilibrium, see \cite{vidyasagar2002nonlinear}. 

The original prey-predator system assumes that $a_{12}$ and $a_{21}$ have different signs. Many higher-dimensional Lotka-Volterra systems have mixed signs for the $a_{ij}$ in fact, due to the applications relevance. Nevertheless those for which all $a_{ij}$ are positive (cooperative systems) and all are negative (competitive systems) have enjoyed significant attention. 

A generalization of the Lotka-Volterra system in \eqref{eq:genLVxi} is proposed in \cite{goh1978sector} as
\begin{equation}\label{eq:gohsystem}
\dot x_i(t)=F_i(x_1(t),x_2(t),\dots,x_n(t))x_i(t)
\end{equation}
for $i=1,2,\dots,n$. The $F_i$ are assumed to be at least two times continuously differentiable, and \eqref{eq:genLVxi} is obtained with the identification
\begin{equation*}
F_i(x_1,x_2,\dots,x_n)=d_i+\sum_{j=1}^n a_{ij}x_j.
\end{equation*}
One can write the vector form of \eqref{eq:gohsystem} as
\begin{equation}\label{eq:gohsystem_vector}
\dot{x}(t) = F(x)x,
\end{equation}
where $F = \diag(F_1(x), \hdots, F_n(x))\in \mathbb{R}^{n\times n}$, and we assume \eqref{eq:gohsystem_vector} is such that $x(0) \in \mathbb{R}^n_{\geq 0}$ implies $x(t) \in \mathbb{R}^n_{\geq 0}$ for all $t\geq 0$. It is obvious that \eqref{eq:gohsystem_vector} has the trivial equilibrium $x = \vect 0_n$. We comment now on other equilibria in the positive orthant. A non-trivial equilibrium $\bar x \in \mathbb{R}^n_{> 0}$ is termed feasible. An equilibrium $\bar x \in \mathbb{R}^n_{\geq 0}$ is termed partially feasible if there exist $i, j \in \{1, \hdots, n\}$ such that $\bar x_i > 0$ and $\bar x_j = 0$. We are interested in establishing a condition for the existence and uniqueness of a feasible equilibrium for \eqref{eq:gohsystem_vector}.

If each $F_i$ in \eqref{eq:gohsystem} has the property that it is positive everywhere on the boundary of the positive orthant where $x_i=0$, there can be no stable equilibria on the boundary, and just inside such a boundary, motions will have a component along the inwardly directed normal to the boundary. If the $F_i$ have the further property that whenever $x\in \mathbb R^n_{\geq 0}$ is such that $\Vert x\Vert > R$ for some constant $R$, there holds $\sum_{i=1}^n F_i(x_1,x_2,\dots,x_n)x_i^2<0$, then the motions of \eqref{eq:gohsystem_vector} will be pointed inwards into $\mathbb R^{n}_{\geq 0}\cap \{x: \Vert x\Vert \leq R\}$ when $\Vert x\Vert=R$.  Note that such an assumption on $F_i$ ensures that the population size $x_i(t)$ of species $i$ does not tend to infinity for any $t \leq \infty$, ensuring the population dynamics are well-defined for all time. By using Proposition~\ref{prop:nagumo} it follows that the interior of the set $\mathbb R^{n}_{> 0}\cap \{x: \Vert x\Vert \leq R\}$ is an invariant of the motion. We now use Theorem~\ref{thm:unique} to derive a sufficient condition for \eqref{eq:gohsystem_vector} to have a unique feasible equilibrium.

\begin{theorem}\label{thm:LV_unique}
	Consider the system \eqref{eq:gohsystem_vector} with $F_i\in\mathcal C^\infty$ for all $i = 1, \hdots, n$. Suppose that there exist constants $R, \epsilon > 0$ such that $\mathcal{W} \triangleq \{x: \Vert x\Vert \leq R, x_i \geq\epsilon\,\forall\,i = 1,\hdots, n\}$ is a positive invariant set and $F(x)$ points inward at every $x\in \partial{\mathcal{W}}$. Then, there exists a feasible equilibrium in $\intr(\mathcal{W})$. Suppose further that for any equilibrium point  $\bar x \in \mathcal{W}$:
	\begin{eqnarray}\label{eq:ineq_jacob_goh}
	\frac{\partial F_i(\bar x)}{\partial x_i}&\leq &a_{ii}<0\\\nonumber
	\left|\frac{\partial F_i(\bar x)}{\partial x_j}\right|&\leq &a_{ij}\,,\forall\,i\neq j
	\end{eqnarray}
	for some constant matrix $A$ for which $-A$ satisfies \ref{C3}. Then, there is a unique feasible equilibrium $x^* \in \intr(\mathcal{W})$, and $x^*$ is locally exponentially stable.  
\end{theorem}
\begin{proof}
	Now because of the $\mathcal{C}^\infty$ assumption on the $F_i$, an argument set out in \cite[Lemma 4.1]{lajmanovich1976SISnetwork} and appealing to Brouwer's Fixed Point Theorem establishes that there is at least one equilibrium point in the convex and compact set $\mathcal{W}$. Notice that all equilibria $\bar x \in \mathcal{W}$ are feasible, satisfying $\bar x > \vect 0_n$. It follows from \eqref{eq:gohsystem_vector} that $F_i(\bar x_1,\bar x_2,\dots, \bar x_n)=0\;\forall i = 1, \hdots, n$.
	
	Let $J_F(x)$ denote the Jacobian of the vector-valued function $\tilde F = [F_1(x), \hdots, F_n(x)]^\top$ evaluated at $x$. The Jacobian of the system \eqref{eq:gohsystem_vector}, denoted $df_x$ to be consistent with the notation in Section~\ref{sec:PH_general}, is computed to be $$df_x = XJ_F(x)+F(x),$$ which at an equilibrium $\bar x \in \mathcal{W}$ is simply 
	\begin{equation}\label{eq:gohsystem_jacob}
	df_{\bar x} = \bar XJ_F(\bar x).
	\end{equation}
	
	First observe that the inequalities in \eqref{eq:ineq_jacob_goh} imply that $a_{ij}\geq 0$ for $i\neq j$ and this means that $-A$ has all off-diagonal entries nonpositive. Lemma~\ref{lem:matrix_conditions} establishes that $-A$ satisfying \ref{C3} (as per the theorem hypothesis) is equivalent to $-A$ being a nonsingular $M$-matrix, and this is in turn equivalent to the existence of a positive diagonal $D$ for which $-AD$ is strictly diagonally dominant. Using Lemma \ref{lem:1} and \eqref{eq:ineq_jacob_goh}, it is then evident that $-J_F(\bar x)D$ is strictly diagonally dominant, and it follows that $-\bar XJ_F(\bar x)D$ is also strictly diagonally dominant. That is, $-\bar XJ_F(\bar x)D$ satisfies \ref{C1}. Lemma~\ref{lem:matrix_conditions} implies that there exists a diagonal positive $C$ for which $$-\left(\bar XJ_F(\bar x)DC+CDJ_F^{\top}(\bar x)\bar X\right)$$ is positive definite (Condition~\ref{C2}). Since $DC=CD$ is a positive definite  matrix, this implies $df_{\bar x} = \bar XJ_F(\bar x)$ has all eigenvalues in the left half plane. The inequality \eqref{eq:ineq_jacob_goh} is assumed to hold for all equilibria $\bar x \in \mathcal{W}$, which implies that $df_{\bar x}$ is Hurwitz for all equilibria $\bar x \in \mathcal{W}$. Theorem~\ref{thm:unique} then establishes that there is in fact a unique feasible equilibrium $x^* \in \intr(\mathcal{W})$, and $x^*$ is locally exponentially stable.
\end{proof}
If in fact \eqref{eq:ineq_jacob_goh} holds for all $x \in \mathcal{W}$, then one has global convergence: $\lim_{t\to\infty} x(t) = x^*$ for all $x(0) \in \mathcal{W}$ exponentially fast. To establish this, first recall that $F_i(x^*) = 0$ for each $i$. Observe that for each $i$ one has
\begin{equation}\label{eq:comparison}
\dot x_i=F_i(x)x_i=\Big(\sum_{j=1}^n\frac{\partial F_i(\tilde x_j)}{\partial x_j}(x_j-x_j^*)\Big)x_i,
\end{equation}
with $\tilde x_j$ taking some value existing by the mean value theorem between $x_j$ and $x_j^*$. Now set $y_i=|x_i-x_i^*|$ for each $i$. It is not hard to verify using \eqref{eq:ineq_jacob_goh} and \eqref{eq:comparison} that 
\begin{align}
\dot y_i & \leq \epsilon\Big(\frac{\partial F_i(\tilde x_i)}{\partial x_i}y_i+\sum_{j\neq i}^n\left|\frac{\partial F_i(\tilde x_j)}{\partial x_j}\right|y_j\Big)\nonumber \\
& \leq \epsilon\Big(a_{ii}y_i+\sum_{j\neq i}^na_{ij}y_j\Big), \label{eq:ineq_y}
\end{align}
where $\epsilon > 0$ is defined in Theorem~\ref{thm:LV_unique}. Let $z = [z_1, \hdots, z_n]^\top$ and $y= [y_1, \hdots, y_n]^\top$. Consider the system
$$ \dot{z}(t) = \epsilon Az(t) \;,\;\;z(0)=y(0).$$
From \eqref{eq:ineq_y}, we obtain that $\dot{y}(t) - \epsilon Ay(t) \leq \dot{z}(t) - \epsilon Az(t)$ for all $t$. The properties of $A$ detailed in Theorem~\ref{thm:LV_unique} allows the main theorem of \cite{walter1971ordinary} to be applied, which establishes that $y(t) \leq z(t)$ for all $t$. Because $A$ is Hurwitz, it follows that $\lim_{t\to\infty} z(t) = \vect 0_n$, and therefore $y_i(t) =|x_i(t)-x_i^*|$ converges to zero, as required. 

\begin{remark}
	It was first established in \cite{goh1978sector} that if \eqref{eq:ineq_jacob_goh} holds for all $x \in \mathcal{W}$, then $\lim_{t\to\infty} x(t) = x^*$ for all $x(0) \in \mathbb{R}^n_{>0} \cap \mathcal{W}$, where $x^*$ is a feasible equilibrium. The uniqueness of $x^*$ was never explicitly proved in \cite{goh1978sector}, but rather implicitly by constructing a complex Lyapunov-like function which simultaneously yielded uniqueness and convergence. Theorem~\ref{thm:LV_unique} relaxes the result of \cite{goh1978sector} in the sense that the inequalities in \eqref{eq:ineq_jacob_goh} are only required to hold when $F_i(x) = 0$ for all $i$ and $x \in \mathcal{W}$, and we explicitly prove the uniqueness property. Moreover, it is conceivable that some forms of \eqref{eq:gohsystem_vector} have a unique feasible equilibrium while still exhibiting chaotic behaviour, limit cycles and other dynamical behaviour associated with Lotka--Volterra systems. We then, separately, recover the global convergence result of \cite{goh1978sector} by a simple argument without requiring Lyapunov-like functions.
\end{remark}



\section{A Discrete-Time Counterpart}\label{sec:DT}

For many processes in the natural and social sciences, both continuous- and discrete-time models exist.  For example, various works have studied discrete-time epidemic models~\cite{allen1994DT_epidemics,prem2020covid} and Lotka--Volterra models~\cite{din2013DT_lotka}. As a consequence, it is natural for systems and control engineers to consider the same set of questions for both continuous- and discrete-time models, such as the uniqueness of the equilibrium. (Whether continuous- or discrete-time models are more appropriate is beyond the scope of this paper, and often depends on the problem context and other factors). In this section, we present a discrete-time counterpart to Theorem~\ref{thm:unique} for the nonlinear system
\begin{equation}\label{eq:general_system_DT}
x(k+1) = G(x(k)), 
\end{equation}
and show an example application on the DeGroot--Friedkin model of a social network \cite{jia2015opinion_SIAM,ye2019DF_journal}. This counterpart result first appeared in \cite[Theorem 3]{anderson2017_lefschetzECC_arxiv}. Note that a point $\bar x$ satisfying $G(\bar x) = \bar x$ is said to be a fixed point of the nonlinear mapping $G$, and $\bar x$ is \textit{an equilibrium} of \eqref{eq:general_system_DT}.  


\begin{theorem}\label{thm:main_lefschetz_result}
	Consider a smooth map $G :X\rightarrow X$ where $X$ is a compact and contractible manifold of finite dimension. Suppose that for all fixed points $\bar x \in X$ of $G$, the eigenvalues\footnote{As defined in Section~\ref{ssec:diff_top}, $dG_x$ is the Jacobian of $G$ in the local coordinates of $x \in X$.} of $dG_{\bar x}$ have magnitude less than 1. Then, $G$ has a unique fixed point $x^* \in X$, and in a local neighbourhood about $x^*$, \eqref{eq:general_system_DT} converges to $x^*$ exponentially fast.
\end{theorem}

Rather than present the proof of the theorem, which can be found in \cite{anderson2017_lefschetzECC_arxiv} and requires some additional knowledge and results on the Lefschetz--Hopf Theorem \cite{hirsch2012differential,armstrong2013basic}, we instead provide some comments relating Theorems~\ref{thm:unique} and \ref{thm:main_lefschetz_result}. First, we note that the existence of a homotopy between $G$ and the identity map is central to the proof of Theorem~\ref{thm:main_lefschetz_result}, as detailed in \cite[Theorem 3]{anderson2017_lefschetzECC_arxiv}. Consequently, \cite{anderson2017_lefschetzECC_arxiv} required $X$ to be a compact, oriented and convex manifold or a convex triangulable space of finite dimension so that a specific homotopy between $G$ and the identity map could be constructed.  However, \cite[Theorem 5.19]{armstrong2013basic} identifies that for any compact and contractible $X$, there exists a homotopy between any map $G : X \to X$ and the identity map. Thus, we can relax the hypothesis in Theorem~\ref{thm:main_lefschetz_result} from \cite[Theorem~3]{anderson2017_lefschetzECC_arxiv} to allow for $X$ to be compact and contractible.

\textit{Mutatis mutandis}, Theorems~\ref{thm:unique} and \ref{thm:main_lefschetz_result} are therefore \textit{equivalent}. The requirement in Theorem~\ref{thm:unique} that $df_{\bar x}$ is Hurwitz for all zeroes $\bar x$ of $f(\cdot)$ in \eqref{eq:general_system} is equivalent to the requirement in Theorem~\ref{thm:main_lefschetz_result} that for all fixed points $\bar x$ of $G$ in \eqref{eq:general_system_DT}, $dG_{\bar x}$ has eigenvalues all with magnitude less than 1. 

\subsection*{Application to the DeGroot--Friedkin Model}

In \cite{anderson2017_lefschetzECC_arxiv}, Theorem~\ref{thm:main_lefschetz_result} is applied to the DeGroot--Friedkin model \cite{jia2015opinion_SIAM}, which describes the evolution of individual self-confidence, $x_i(k)$, as a social network of $n \geq 3$ individuals discusses a sequence of issues, $k = 0, 1, 2, \hdots$. We provide a summary of the application here. The map $G$ in question is: 
\begin{align}\label{eq:map_F_DF}
G(  x(k) ) =   \frac{1}{\sum_{i=1}^n \frac{\gamma_i}{1- x_i(k)}} \begin{bmatrix} \frac{\gamma_1}{1-x_1(k)} \\ \vdots \\ \frac{\gamma_n}{1-x_n(k)} \end{bmatrix}      
\end{align} 
where $\gamma_i \in (0, 0.5)$, and $\sum_{i=1}^n \gamma_i = 1$. The compact, convex and oriented manifold of interest is $$\wt{\Delta}_n = \{x_i : \sum_{i=1}^n x_i = 1, 0 < \delta \leq x_i \leq 1-\delta\},$$ where $\delta>0$ is arbitrarily small. One can regard $\wt{\Delta}_n$ as a compact subset in the interior of the $n-1$-dimensional unit simplex, and it can be shown that $G : \wt{\Delta}_n \to \wt{\Delta}_n$ for sufficiently small $\delta$ \cite{jia2015opinion_SIAM,ye2019DF_journal}.

Now, the $G$ in \eqref{eq:map_F_DF} is given with coordinates in $\mathbb{R}^n$, whereas $\tilde{\Delta}_n$ is a manifold of dimension $n-1$. Thus, an appropriate $\mathbb{R}^{n-1}$ coordinate basis is proposed in \cite{anderson2017_lefschetzECC_arxiv}, with an associated map $\tilde G$ \textit{on the manifold} $\wt{\Delta}_n$. Then, \cite{anderson2017_lefschetzECC_arxiv} establishes that the eigenvalues of $d\tilde G_{\bar x}$ at every fixed point $\bar x \in \wt{\Delta}_n$ are all of magnitude less than 1. This is done by showing that the eigenvalues of $d\tilde G_{\bar x}$ are a subset of the eigenvalues of a Laplacian matrix $\mathcal{L}$ associated with a strongly connected graph. It is well known that such a Laplacian has a single zero eigenvalue and all other eigenvalues have positive real part. In fact, \cite{anderson2017_lefschetzECC_arxiv} shows the particular $\mathcal{L}$ has all real eigenvalues, and its trace is 1. Since $n \geq 3$, it immediately follows that all eigenvalues of $\mathcal{L}$, and by implication all eigenvalues of $d\tilde G_{\bar x}$ are less than 1 in magnitude. 

One can then apply Theorem~\ref{thm:main_lefschetz_result} to establish that $\tilde G$ has a unique fixed point $x^*$ in $\wt{\Delta}_n$ (and consequently the $G$ in \eqref{eq:map_F_DF}), and $x^*$ is locally exponentially stable for the system \eqref{eq:general_system_DT}. We refer the reader to \cite[Theorem 4]{anderson2017_lefschetzECC_arxiv} for the details. We conclude by remarking that the first proof of the uniqueness of $x^*$ in $\wt{\Delta}_n$ for $G$ in \eqref{eq:map_F_DF} required extensive and complex algebraic manipulations, see \cite[Appendix F]{jia2015opinion_SIAM}. In comparison, the calculations required to establish the uniqueness of $x^*$ in $\wt{\Delta}_n$ for $G$ in \eqref{eq:map_F_DF} using Theorem~\ref{thm:main_lefschetz_result} are greatly simplified, and may continue to hold for generalisations of \eqref{eq:map_F_DF} as studied in \cite{ye2019_ECC_DF_Distort}.

\section{Conclusions}\label{sec:conclusion}
We have used the Poincar\'e--Hopf Theorem to prove that a nonlinear dynamical system has a unique equilibrium (that is actually locally exponentially stable) if inside a compact and contractible manifold, its Jacobian at every possible equilibrium is Hurwitz. We illustrated the method by applying it to analyse the established deterministic SIS networked model, and an extension that introduces decentralised controllers, rendering the system no longer quadratic. We proved a general impossibility result: if the uncontrolled system has a unique endemic equilibrium, then the controlled system also has a unique endemic equilibrium, which is locally exponentially stable. I.e., the controllers can never globally drive the networked system to the healthy equilibrium. A stronger, almost global convergence result was obtained by extending a result from monotone dynamical systems theory, with the extension relying on the fact that the endemic equilibrium was unique. A generalised nonlinear Lotka--Volterra model was also analysed. Last, a counterpart sufficient condition was presented for a nonlinear discrete-time system to have a unique equilibrium in a compact and contractible manifold.
For future work, we hope expand the analysis framework presented in this paper, and identify further applications, especially focusing on various models within the natural sciences. This includes the introduction of control to other existing epidemic models.

\appendices

\section{Monotone Systems}\label{app:monotone}
A simple introduction to monotone systems is provided, sufficient for the purposes of this paper. A general convergence result is then developed, to be used in Section~\ref{ssec:feedback}. For details, the reader is referred to \cite{smith2008monotone_book,smith1988monotone_survey}. We impose slightly more restrictive conditions than in \cite{smith2008monotone_book,smith1988monotone_survey} for the purposes of maintaining the clarity and simplicity of this section.

To begin, let $m = [m_1, \hdots, m_n]^\top \in \mathbb{R}^n$, with $m_i \in \{0, 1\}$ for $i = 1, \hdots, n$. Then, an orthant of $\mathbb R^n$ can be defined as 
\begin{equation}K_m = \{x\in \mathbb{R}^n : (-1)^{m_i} x_i \geq 0,\,\forall\,i\in \{1, \hdots, n\}\}.
\end{equation}
For a given orthant $K_m \in \mathbb{R}^n$, we write $x \leq_{K_m} y$ and $x <_{K_m} y$ if $y-x \in K_m$ and $y - x \in \intr(K_m)$, respectively.

We consider the system \eqref{eq:general_system} on a convex, open set $U \subseteq \mathbb{R}^n$, and assume that $f$ is sufficiently smooth such that $df_x$ exists for all $x\in U$, and the solution $x(t)$ is unique for every initial condition in $U$. We use $\phi_t(x_0)$ to denote the solution $x(t)$ of \eqref{eq:general_system} with $x(0) = x_0$. If whenever $x_0 ,y_0 \in U$, satisfying $x_0 \leq_{K_m} y_0$, implies $\phi_t(x_0) \leq_{K_m} \phi_t(y_0)$ for all $t\geq 0$ for which both $\phi_t(x_0)$ and $\phi_t(y_0)$ are defined, then the system \eqref{eq:general_system} is said to be a type $K_m$ monotone system and the solution operator $\phi_t(x_0)$ of \eqref{eq:general_system} is said to preserve the partial ordering $\leq_{K_m}$ for $t\geq 0$. The following is a necessary and sufficient condition for \eqref{eq:general_system} to be type $K_m$ monotone, and focuses on the Jacobian $df_x$ of $f(\cdot)$ in \eqref{eq:general_system}.

\begin{lemma}[{Kamke--M\"{u}ller Condition \cite[Lemma 2.1]{smith1988monotone_survey}}]\label{lem:monotone}
	Suppose that $f$ is of class $C^1$ in $U$, where $U$ is open and convex in $\mathbb{R}^n$. Then, $\phi_t(x_0)$ of \eqref{eq:general_system} preserves the partial ordering $\leq_{K_m}$ for $t \geq 0$ if and only if $P_m df_x P_m$ has all off-diagonal entries nonnegative  for every $x \in U$, where $P_m = \diag((-1)^{m_1}, \hdots, (-1)^{m_n})$.
\end{lemma}

Many results exist establishing convergence of type $K_m$ monotone systems, with various additional assumptions imposed. Here, we state one which has some stricter assumptions, and then extend it for use in our analysis in Section~\ref{ssec:feedback}. Let $E$ denote the set of equilibria of \eqref{eq:general_system}, and for an equilibrium $e \in E$, the basin of attraction of $e$ is denoted by $B(e)$. We say \eqref{eq:general_system} is an \textit{irreducible} type $K_m$ monotone system if $df_x$ is irreducible for all $x \in U$. 

\begin{lemma}[\hspace{-0.6pt}{\cite[Theorem 2.6]{smith1988monotone_survey}}]\label{lem:monotone_converge}
	Let $\mathcal{M}$ be an open, bounded, and positively invariant set for an irreducible type $K_m$ monotone system \eqref{eq:general_system}. Suppose the closure of $\mathcal{M}$, denoted by $\overline{\mathcal{M}}$, contains a finite number of equilibria.
	Then,
	\begin{equation}
	\bigcup_{e\in E\cap \overline{\mathcal{M}}} \intr(B(e)) \cap \overline{\mathcal{M}}
	\end{equation}
	is open and dense in $\mathcal{M}$.
\end{lemma}
A set $S\subset A$ is dense in $A$ if every point $x \in A$ is either in $S$ or in the closure of $S$. Thus, Lemma~\ref{lem:monotone_converge} states that for an irreducible type $K_m$ monotone system \eqref{eq:general_system}, the system converges to an equilibrium $e \in E\cap \overline{\mathcal{M}}$ for almost all initial conditions in $\mathcal{M}$. There are at most a finite number of nonattractive limit cycles. A stronger result is available, appearing in \cite[Theorem D]{ji1994global_cooperative}, and presented below with a different, simpler proof. The result will be used in Section~\ref{sec:epidemics}.
\begin{proposition}[{c.f. \cite[Theorem D]{ji1994global_cooperative}}]\label{prop:conv}
	Let $\mathcal{M}$ be an open, bounded, convex, and positively invariant set for an irreducible type $K_m$ monotone system \eqref{eq:general_system}. Suppose there is a unique equilibrium $e^* \in  \mathcal{M}$ and no equilibrium in $\overline{\mathcal{M}}\setminus \mathcal{M}$. Then, convergence to $e^*$ occurs for every initial condition in $\mathcal{M}$.
\end{proposition}	
\begin{proof}
	 First, we remark that an irreducible type $K_m$ monotone system \eqref{eq:general_system} enjoys a stronger monotonicity property; for any $x_1, x_2 \in \mathcal{M}$, one has that \mbox{$x_1 <_{K_m} x_2 \Rightarrow \phi_t(x_1) <_{K_m} \phi_t(x_2)$} for all $t > 0$ \cite{smith2008monotone_book}.
	 
	 In light of Lemma~\ref{lem:monotone_converge}, the proposition is proved if we establish that there does not exist a limit cycle. We argue by contradiction.
	 Let $a$ be a point on such a limit cycle of \eqref{eq:general_system}. Pick two points $\underline{a} \in \mathcal{M}$ and $\bar a \in \mathcal{M}$ satisfying $\underline{a} <_{K_m} a <_{K_m} \bar a$, and observe that there exist two sufficiently small balls $\mathcal{B}_1$ and $\mathcal{B}_2$ surrounding $\underline{a}$ and $\bar a$, respectively, which neither intersect the boundary of $\mathcal{M}$ nor contain $a$, and every point $x \in \mathcal{B}_1$ and $y \in \mathcal{B}_2$ obey $x <_{K_m} a <_{K_m} y$. Since almost every point in $\mathcal{B}_1$ is not in a nonattractive limit cycle, there exists an $x_1 \in \mathcal{B}_1$ such that $\lim_{t\to\infty} \phi_t (x_1) = e^*$. Similarly, there exists a $y_2 \in \mathcal{B}_2$ such that $\lim_{t\to\infty} \phi_t (y_2) = e^*$.  Because $x_1 <_{K_m} a <_{K_m} y_2$, it follows that  $\phi_t(x_1) <_{K_m} \phi_t(a) <_{K_m} \phi_t(y_2)$. Recalling that $\lim_{t\to\infty} \phi_t (x_1) = e^*$ and  $\lim_{t\to\infty} \phi_t (y_2) = e^*$ yields $\lim_{t\to\infty} \phi_t(a) = e^*$. However, this contradicts the assumption that $a$ is a point on a nonattractive limit cycle. 
\end{proof} 


\section{Proof of Lemma~\ref{lem:positive_invariance}}\label{app:lem_pf}
To begin, we prove the first part of the lemma statement. Fixing $i \in \{1, \hdots, n\}$, we need to show that $-\mathbf{e}_i^\top \dot{x} = -\dot{x}_i < 0$ for $x \in P_i$. Now, $x$ satisfies $x_i = \epsilon y_i$ and for $j\neq i$, we have $x_j = \epsilon y_j + z_j$ for some $z_j \geq 0$. From \eqref{eq:SIS_contact_node}, we find
\begin{align}
	\dot{x}_i  
	= & \epsilon\Big(-d_iy_i + \sum_{j=1}^n b_{ij} y_j\Big) - \epsilon^2 y_i \sum_{j=1}^n b_{ij} y_j \nonumber \\
	& \qquad+ (1-\epsilon y_i)\sum_{j=1}^n b_{ij} z_j.
\end{align}
The identity $(-D + B)y = \phi y$ implies $\phi y_i = - d_iy_i + \sum_{j=1}^n b_{ij} y_j$, and substituting this into the above and rearranging yields
\begin{align}\label{eq:dot_x_lemma}
	\dot{x}_i = \epsilon y_i\Big(\phi - \epsilon \sum_{j=1}^n b_{ij} y_j\Big) + (1-\epsilon y_i)\sum_{j=1}^n b_{ij} z_j.
\end{align}
Since $\phi> 0$ is constant and $y_k \leq 1$ for all $k$, there exists a sufficiently small $\epsilon_i > 0$ such that for all $\epsilon \leq \epsilon_i$, the first and second summand on the right of \eqref{eq:dot_x_lemma} are positive and nonnegative, respectively. It follows from \eqref{eq:dot_x_lemma} that  $-\mathbf{e}_i^\top \dot{x} < 0$ for all $x \in P_i$. Repeating the analysis for $i = 1, \hdots n$, it is clear that \eqref{eq:boundary_point1} holds for all $i = 1, \hdots n$, for any positive $\epsilon \leq \min_{i} \epsilon_i$.

Next, fix $i \in \{1, \hdots, n\}$, and consider a point $x \in Q_i$. Now, $\mathbf{e}_i^\top \dot{x} = \dot{x}_i$, and \eqref{eq:SIS_contact_node} yields $	\dot{x}_i = -d_i < 0$, and this holds for all $i = 1, \hdots n$, and thus \eqref{eq:boundary_point2} holds. It follows from Proposition~\ref{prop:nagumo} that for $0 < \epsilon \leq \min_{i} \epsilon_i$, $\mathcal{M}_\epsilon$ is a positive invariant set of \eqref{eq:SIS_contact_network}. \eqref{eq:boundary_point} shows that $\partial\mathcal{M}_{\epsilon}$ is not an invariant set of \eqref{eq:SIS_contact_network}, which implies that $\intr(\mathcal{M}_\epsilon)$ is also a positive invariant set of \eqref{eq:SIS_contact_network}.

To prove the second part of the lemma, consider a point $x(t) \in \partial\Xi_n\setminus \vect 0_n$, at some time $t\geq 0$. If $x_i(t) = 1$ for some $i \in \{1, \hdots, n\}$, then \eqref{eq:SIS_contact_node} yields $\dot{x}_i = -d_i < 0$. Thus, if $x(t) > \vect 0_n$, then obviously $x(t+\kappa_1) \in \mathcal{M}_{\epsilon_{1,x}}$ for some sufficiently small positive $\kappa_1$ and $\epsilon_{1,x}$.

Let us suppose then, that $x(t) \in \partial\Xi_n\setminus \vect 0_n$ has at least one zero entry. Define the set $\mathcal{U}_t \triangleq \{i : x_i(t) = 0, i \in \{1, \hdots, n\}\}$. The lemma hypothesises that $\mathcal{G}$ is strongly connected, which implies that there exists a $k \in \mathcal{U}_t$ such that $x_j(t) > 0$ for some $j \in \mathcal{N}_k$. \eqref{eq:SIS_contact_node} yields $\dot{x}_k = \sum_{l\in \mathcal{N}_k} b_{lk} x_l(t) \geq b_{jk} x_j(t) > 0$. This analysis can be repeated to show that there exists a finite $\kappa_2$ such that $\mathcal{U}_{t+\kappa_2}$ is empty.  It follows that $x(t+\kappa_2) \in \mathcal{M}_{\epsilon_{2,x}}$ for some sufficiently small $\epsilon_{2,x}$.

Since $\partial\Xi_n\setminus \vect 0_n$ is bounded, there exists a finite $\bar \kappa$ and sufficiently small $\epsilon \in (0, \min_{x}\{\epsilon_{1, x}, \epsilon_{2, x}\}]$ such that $x(\bar\kappa) \in \mathcal{M}_{\epsilon}$ for all $x(0) \in \partial\Xi_n\setminus \vect 0_n$. Since $\epsilon$ can be taken to be arbitrarily small, it is also clear that any nonzero equilibrium $x$ must satisfy $x\in \intr(\Xi_n)$. \hfill $\qed$

\section{Proof of Lemma~\ref{lem:new_eqb}}\label{app:lem_pf_eqb}
Let $\bar x^* \in \intr(\Xi_n)$ denote the endemic equilibrium of \eqref{eq:network_SIS_contact_control}, and $\bar H = \diag(h_1(\bar x_1^*), \hdots, h_n(\bar x_n^*))$. Note the assumption that there exists a $j$ such that $x_j > 0 \Rightarrow h_j(x_j) > 0$ implies $\bar H$ is not only a nonnegative diagonal matrix, but has at least one positive entry. Observe that \eqref{eq:network_SIS_contact_control} and 
\begin{equation}\label{eq:pf_1}
\dot x=(-D-\bar H+(I_n-X)B)x
\end{equation}
have the same positive equilibrium $\bar x^*$. By arguments introduced previously, we know that $\bar x^*$ is the only positive equilibrium for \eqref{eq:pf_1}. Hence $s(-D-\bar H+B)>0$ by Proposition~\ref{prop:SIS_network_convergence}. Let $c > 0$ be a sufficiently large constant such that $P(\alpha) = cI_n -D-\alpha\bar H+B$ is nonnegative for all $\alpha \in [0,1]$ (note that $P(\alpha)$ is irreducible). The Perron--Frobenius Theorem \cite{berman1979nonnegative_matrices} yields $s(P(\alpha)) = c + s(-D-\alpha\bar H+B)$. For $\alpha_1 > \alpha_2$, one concludes that $P(\alpha_2)$ is equal to $P(\alpha_1)$ plus some nonnegative diagonal entries (of which at least one is positive), and \cite[Corollary 2.1.5]{berman1979nonnegative_matrices} yields that $s(P(\alpha_1)) < s(P(\alpha_2))$. This implies that $s(-D-\alpha_1\bar H+B) < s(-D-\alpha_2\bar H+B)$. It follows that $s(-D-\alpha\bar H+B)>0$ for all $\alpha\in[0,1]$. Hence the system
\begin{equation}\label{eq:pf_2}
\dot x=(-D-\alpha\bar H+(I_n-X)B)x
\end{equation}
has for all $\alpha\in [0,1]$ a unique equilibrium in $\intr(\Xi_n)$, call it $\bar x_{\alpha}$. Notice that $\bar x_0= x^*$ (the endemic equilibrium of the uncontrolled system \eqref{eq:SIS_contact_network}) and $\bar x_1=\bar x^*$. 

Since $\bar X_{\alpha} = \diag\big((\bar x_\alpha)_1, \hdots, (\bar x_\alpha)_n\big)$ is diagonal, there holds $\big(\frac{d}{d\alpha}\bar X_{\alpha}\big)B\bar x_\alpha = \tilde{B}_\alpha\frac{d\bar x_{\alpha}}{d\alpha}$, where $\tilde{B}_\alpha = \diag\big((B\bar x_\alpha)_1, \hdots, (B\bar x_\alpha)_n\big)$. Then, differentiating $$(-D-\alpha\bar H+(I_n-\bar X_{\alpha})B)\bar x_{\alpha}=0$$ with respect to $\alpha$ yields after some rearranging:
\begin{equation}
-\underbrace{(D+\alpha\bar H-(I_n-\bar X_\alpha)B+\tilde B_{\alpha}))}_{K_{\alpha}}\frac{d\bar x_{\alpha}}{d\alpha}=\bar H\bar x_{\alpha}.
\end{equation}
Using arguments similar to those laid out in the proof of Theorem~\ref{thm:SIS_ph_unique} (see Eq.~\ref{eq:endemic_equib} and below) it can be shown that  $K_\alpha$ is an irreducible, nonsingular $M$-matrix. \cite[Theorem 2.7]{berman1979nonnegative_matrices} yields that $K_{\alpha}^{-1} > \mat 0_{n\times n}$. Next, one can verify that $\bar H\bar x_{\alpha} \geq \vect 0_{n\times n}$ has at least one positive entry since $\bar H$ has at least one positive entry and $\bar x_{\alpha} > \vect 0_n$. This means that $$\frac{d\bar x_{\alpha}}{d\alpha} = -K_{\alpha}^{-1}\bar H\bar x_{\alpha} <\vect 0_n.$$ Integration yields $\bar x_0 = x^*>\bar x^* = \bar x_1$, as claimed.  \hfill $\qed$

\section{Properties of the Tangent Cone}\label{app:tangent_cone}
We now prove the claims below \eqref{eq:tangent_cone}, that identify the relationship between $\mathcal{Z}_{\mathcal{M}}(x)$ and $T_x\mathcal{M}$, for a smooth and compact $m$-dimensional manifold $\mathcal{M}$ embedded in $\mathbb{R}^n$, with $m \leq n$.

To begin, we state the following well known result, which can be established from the material in \cite[pg.~8--10]{guillemin2010differential}.

\begin{proposition}
	The tangent space at a point $x$ in a $m$-dimensional manifold $\mathcal{M}$ embedded in $\mathbb R^n$ is an $m$-dimensional Euclidean space (``hyperplane'' in the case of $m<n$)  in $\mathbb R^n$. 
\end{proposition}

The following lemma, which is intuitively obvious, is also required.

\begin{lemma}\label{lem:dist}
	Let $y_1,y_2$ be two points outside a compact set $X$. Then there holds
	\[
	{\rm{dist}}(y_1,X)\leq\|y_1-y_2\| +{\rm{dist}}(y_2,X)
	\]
\end{lemma}
\begin{proof}Consider two paths and the associated distances from $y_1$ to $X$. The first is that via the shortest distance path, which is ${\rm{dist}}(y_1,X)$, and the second is that via the shortest path which is constrained to pass through $y_2$. The second is never shorter than the first, which is the content of the inequality.
\end{proof}

\subsection{Manifolds Without Boundary}
In this subsection, we shall assume that the manifold $\mathcal{M}$ does \textit{not} have a boundary. We shall treat the case of manifolds with a boundary in the next subsection. The tangent cone of $\mathcal{M}$ at $x$ is defined in \eqref{eq:tangent_cone}. We shall distinguish  $x \notin \mathcal{M}$ and $x\in \mathcal{M}$, treating these in order. The first result is asserted in \cite{blanchini1999set_invariance} for $m=n$, but that reference does not consider $m < n$.

\begin{proposition}\label{prop:notin}
	With notation as above, suppose $x\notin \mathcal{M}$. Then $\mathcal Z_X(x)= \emptyset$.
\end{proposition}
\begin{proof}Observe first that by the compactness of $\mathcal{M}$, $\inf_{y\in \mathcal{M}}\|x-y\|$ is attained at some $\bar y\in \mathcal{M}$. Moreover $\|x-\bar y\|$ is necessarily nonzero, assuming a value $d$ say, since otherwise we would have $x\in \mathcal{M}$, a contradiction. Now use Lemma~\ref{lem:dist}, identifying $y_1$ and $y_2$ with $x$ and $x+hz$ respectively. We have
	\[
	{\rm{dist}}(x+hz,\mathcal{M})\geq{\rm{dist}}(x,\mathcal{M})-|h|\|z\|=d-|h|\|z\|.
	\] 
	It follows that for no $z$, can we have
	
	\[
	\lim_{h\rightarrow 0}\inf\frac{\text{dist}(x+hz,\mathcal{M})}{h}=0.
	\]
	This completes the proof.
\end{proof}

We now turn our focus onto the case $x\in \mathcal{M}$. Our aim is to show that $\mathcal Z_\mathcal{M}(x)=T_x\mathcal{M}$ for all $x\in X$, thus linking the concepts of a tangent cone and a tangent space. We first of all establish a one-way inclusion.

\begin{proposition}\label{prop:oneway}
	With notation as above and $x \in \mathcal{M}$, there holds $T_x\mathcal{M}\subset\mathcal Z_X(x)$.
\end{proposition}
\begin{proof}Suppose $z \in T_x\mathcal{M}$ but is otherwise arbitrary. Let the smooth function $\phi:U\rightarrow V$ be a local parametrisation around $x$ with $\phi(0)=x$. 
	Then there exists a $u \in \mathbb R^m$ such that $d\phi_0(u)=z$. 
	(In a particular coordinate basis, $d\phi_0(u)$ is simply equal to $Ju$, where $J = d\phi_0$ is the Jacobian matrix of $\phi$ evaluated at $0$, and $u$ is a vector.)  Observe then that 
	\[
	hz=hd\phi_0(u)=d\phi_0(hu).
	\]
	Then for $h$ sufficiently small, 
	\[
	\phi(hu)=\phi(0)+d\phi_0(hu)+o(h)
	\]
	by the ``best linear approximation'' interpretation of $d\phi_0(u)$, c.f.~\cite[pg. 9]{guillemin2010differential}. 
	This means that
	\[
	\phi(hu)=x+d\phi_0(hu)+o(h)
	= x+hz+o(h)
	\]
	Observe that $\phi(hu)$ lies in $\mathcal{M}$. So the distance from $x+hz$ to $\mathcal{M}$ is overbounded by $o(h)$
	This distance divided by $h$ is therefore overbounded by $h^{-1}o(h)$. Hence as $h$ goes to zero, this bound goes to zero. This implies that $z$ is in $\mathcal Z_\mathcal{M}(x)$ by \eqref{eq:tangent_cone}. But $z$ was an arbitrary vector in $T_x\mathcal{M}$. This establishes that $T_x\mathcal{M}$ is contained in $\mathcal Z_\mathcal{M}(x)$.
\end{proof}

It remains to prove the reverse inclusion. In the case of $m=n$, the tangent space corresponds to the whole space in which $\mathcal{M}$ is embedded, and so $\mathbb R^n = T_x\mathcal{M} \subset \mathcal Z_\mathcal{M}(x)\subset \mathbb R^n$, which establishes the claim that the tangent cone and the tangent space coincide. Thus, we now need to focus just on the case $m<n$.

\begin{proposition}\label{prop:otherway}
	With notation as above, and $x\in\mathcal{M}$, there holds $\mathcal Z_\mathcal{M}(x)\subset T_x\mathcal{M}$
\end{proposition}
\begin{proof}
	Suppose, in order to obtain a contradiction, that $ z \in \mathcal Z_\mathcal{M}(x)$, while also $z\notin T_x\mathcal{M}$. Set $z=z_1+z_2$, where $z_1\in T_x\mathcal{M}$ and $z_2\neq 0$ is orthogonal to $z_1$. Suppose that $\inf_{y\in X}\|x+hz-y\|$ is attained at $\bar y (h,z)=\phi(\bar u)$ for some $\bar u\in U$, depending on $h$ and $z$. Such a $\bar u$ exists if $|h|$ is sufficiently small. Now, observe that
	\begin{align*}
		\|x+hz-\phi(\bar u)\|& =\|x+hz-(\phi(0)+d\phi_0(\bar u))\|+o(h)\\
		& =\|hz-d\phi_0(\bar u)\|+o(h).
	\end{align*}
	Note that the second equality was obtained by using the `best linear approximation', c.f.~\cite[pg.~9]{guillemin2010differential}. Now, since \mbox{$d\phi_0(\bar u)\in T_x\mathcal{M}$}, it follows that, irrespective of the particular $h$ and thus the particular $\bar u$, there holds
	\[
	\|hz-d\phi_0(\bar u)\|\geq \inf_{\bar z\in T_x\mathcal{M}}\|hz-\bar z\|=\|hz_2\|.
	\]
	Thus, 
	\begin{align*}
		\text{dist}(x+hz,\mathcal{M})& =\inf_{y\in \mathcal{M}}\|x+hz-y\|\\
		& =\|x+hz-\phi(\bar u)\|\geq |h|\|z_2\|+o(h)
	\end{align*}
	From this, it follows that $z\notin \mathcal Z_\mathcal{M}(x)$ because $z_2\neq 0$, which is a contradiction.
\end{proof}

\subsection{Manifolds With Boundary}\label{app:boundary}
We now assume that $\mathcal{M}$, in addition to the properties detailed and the start of Appendix~\ref{app:tangent_cone}, also has a boundary $\partial\mathcal{M}$, and that boundary is also smooth. It is necessary at the outset to modify the definition of the tangent cone in \eqref{eq:tangent_cone} to distinguish directions associated with any value of $x$ on $\partial M$, so that it is defined using one-sided limits:
\begin{equation}\label{eq:tangent_cone_2}
	\mathcal Z_{\mathcal Q}(x)=\left\{z \in \mathbb R^n:\lim_{h\to 0^+}\inf\frac{\text{dist}(x+hz,\mathcal Q)}{h}=0\right\}
\end{equation}
For manifolds without a boundary, the two definitions are evidently equivalent. 

It is straightforward to see that key results of the previous subsection remain valid, with essentially unchanged proofs. In fact:
\begin{enumerate}
	\item 
	Proposition \ref{prop:notin} remains valid;
	\item 
	When $x\in \mathcal M \setminus \partial \mathcal M$, then $T_x\mathcal M=\mathcal T_\mathcal M(x)$; equivalently Propositions \ref{prop:oneway} and \ref{prop:otherway} continue to hold.  
\end{enumerate}

The remaining interest is in the case when $x\in\partial \mathcal M$. In the case where $m=n$ and assuming the boundary is smooth at $x$, the tangent cone is just the tangent half-space but shifted to the origin, as observed explicitly in the reference \cite{blanchini1999set_invariance}. This reference does not actually prove the result, but it is easily obtained using the methods of the previous section, using the adjusted notion of the tangent cone in \eqref{eq:tangent_cone_2}. And indeed, this continues to hold for the case $m\leq n$, rather than just $m=n$.

\ifCLASSOPTIONcaptionsoff
  \newpage
\fi



%
%
%

\bibliographystyle{IEEEtran}
\bibliography{MYE_ANU}

\begin{thebibliography}{10}
\providecommand{\url}[1]{#1}
\csname url@samestyle\endcsname
\providecommand{\newblock}{\relax}
\providecommand{\bibinfo}[2]{#2}
\providecommand{\BIBentrySTDinterwordspacing}{\spaceskip=0pt\relax}
\providecommand{\BIBentryALTinterwordstretchfactor}{4}
\providecommand{\BIBentryALTinterwordspacing}{\spaceskip=\fontdimen2\font plus
\BIBentryALTinterwordstretchfactor\fontdimen3\font minus
  \fontdimen4\font\relax}
\providecommand{\BIBforeignlanguage}[2]{{%
\expandafter\ifx\csname l@#1\endcsname\relax
\typeout{** WARNING: IEEEtran.bst: No hyphenation pattern has been}%
\typeout{** loaded for the language `#1'. Using the pattern for}%
\typeout{** the default language instead.}%
\else
\language=\csname l@#1\endcsname
\fi
#2}}
\providecommand{\BIBdecl}{\relax}
\BIBdecl

\bibitem{shuai2013epidemic_lyapunov}
Z.~Shuai and P.~van~den Driessche, ``Global stability of infectious disease
  models using lyapunov functions,'' \emph{SIAM Journal on Applied
  Mathematics}, vol.~73, no.~4, pp. 1513--1532, 2013.

\bibitem{li2010global_network}
M.~Y. Li and Z.~Shuai, ``Global-stability problem for coupled systems of
  differential equations on networks,'' \emph{Journal of Differential
  Equations}, vol. 248, no.~1, pp. 1--20, 2010.

\bibitem{takeuchi1996global}
Y.~Takeuchi, \emph{{Global Dynamical Properties of Lotka--Volterra
  Systems}}.\hskip 1em plus 0.5em minus 0.4em\relax World Scientific, 1996.

\bibitem{smale1976differential}
S.~Smale, ``{On the Differential Equations of Species in Competition},''
  \emph{Journal of Mathematical Biology}, vol.~3, no.~1, pp. 5--7, 1976.

\bibitem{milnor1997topology}
J.~W. Milnor, \emph{{Topology from the Differentiable Viewpoint}}.\hskip 1em
  plus 0.5em minus 0.4em\relax Princeton University Press, 1997.

\bibitem{belabbas2013_formationstable}
M.~A. {Belabbas}, ``{On Global Stability of Planar Formations},'' \emph{IEEE
  Transactions on Automatic Control}, vol.~58, no.~8, pp. 2148--2153, Aug 2013.

\bibitem{moghadas2004_PHepidemics}
S.~M. Moghadas, ``Analysis of an epidemic model with bistable equilibria using
  the poincar{\'e} index,'' \emph{Applied Mathematics and Computation}, vol.
  149, no.~3, pp. 689--702, 2004.

\bibitem{simsek2007generalized_PH}
A.~Simsek, A.~Ozdaglar, and D.~Acemoglu, ``{Generalized Poincare-Hopf Theorem
  for compact Nonsmooth Regions},'' \emph{Mathematics of Operations Research},
  vol.~32, no.~1, pp. 193--214, 2007.

\bibitem{tang2007_PHcongestion}
A.~{Tang}, J.~{Wang}, S.~H. {Low}, and M.~{Chiang}, ``{Equilibrium of
  Heterogeneous Congestion Control: Existence and Uniqueness},'' \emph{IEEE/ACM
  Transactions on Networking}, vol.~15, no.~4, pp. 824--837, Aug 2007.

\bibitem{christensen2017PH_potentialgames}
F.~Christensen, ``A necessary and sufficient condition for a unique maximum
  with an application to potential games,'' \emph{Economics Letters}, vol. 161,
  pp. 120--123, 2017.

\bibitem{henshaw2015_SVIR_PoincareHopf}
S.~Henshaw and C.~C. McCluskey, ``Global stability of a vaccination model with
  immigration,'' \emph{Electronic Journal of Dierential Equations}, vol.~92,
  pp. 1--10, 2015.

\bibitem{konovalov2010price_PH}
A.~Konovalov and Z.~S{\'a}ndor, ``On price equilibrium with multi-product
  firms,'' \emph{Economic Theory}, vol.~44, no.~2, pp. 271--292, 2010.

\bibitem{rijk1983equilibrium}
F.~J.~A. Rijk and A.~C.~F. Vorst, ``Equilibrium points in an urban retail model
  and their connection with dynamical systems,'' \emph{Regional Science and
  Urban Economics}, vol.~13, no.~3, pp. 383--399, 1983.

\bibitem{varian1975ph_unique}
H.~R. Varian, ``A third remark on the number of equilibria of an economy,''
  \emph{Econometrica}, vol.~43, no. 5-6, p. 985, 1975.

\bibitem{lajmanovich1976SISnetwork}
A.~Lajmanovich and J.~A. Yorke, ``{A Deterministic Model for Gonorrhea in a
  Nonhomogeneous Population},'' \emph{Mathematical Biosciences}, vol.~28, no.
  3-4, pp. 221--236, 1976.

\bibitem{fall2007SIS_model}
A.~Fall, A.~Iggidr, G.~Sallet, and J.-J. Tewa, ``{Epidemiological Models and
  Lyapunov Functions},'' \emph{Mathematical Modelling of Natural Phenomena},
  vol.~2, no.~1, pp. 62--83, 2007.

\bibitem{khanafer2016SIS_positivesystems}
A.~Khanafer, T.~Ba{\c{s}}ar, and B.~Gharesifard, ``Stability of epidemic models
  over directed graphs: A positive systems approach,'' \emph{Automatica},
  vol.~74, pp. 126--134, 2016.

\bibitem{mei2017epidemics_review}
W.~Mei, S.~Mohagheghi, S.~Zampieri, and F.~Bullo, ``On the dynamics of
  deterministic epidemic propagation over networks,'' \emph{Annual Reviews in
  Control}, vol.~44, pp. 116--128, 2017.

\bibitem{vanMeighem2009_virus}
P.~V. Mieghem, J.~Omic, and R.~Kooij, ``Virus spread in networks,''
  \emph{IEEE/ACM Transactions on Networking}, vol.~17, no.~1, pp. 1--14, 2009.

\bibitem{smith1988monotone_survey}
H.~L. Smith, ``{Systems of Ordinary Differential Equations Which Generate an
  Order Preserving Flow. A Survey of Results},'' \emph{SIAM Review}, vol.~30,
  no.~1, pp. 87--113, 1988.

\bibitem{smith2008monotone_book}
------, \emph{{{Monotone Dynamical Systems: An Introduction to the Theory of
  Competitive and Cooperative Systems}}}.\hskip 1em plus 0.5em minus
  0.4em\relax American Mathematical Society: Providence, Rhode Island, 2008,
  vol.~41.

\bibitem{liu2019bivirus}
J.~Liu, P.~E. Par{\'e}, A.~Nedich, C.~Y. Tang, C.~L. Beck, and T.~Ba\c{s}ar,
  ``{Analysis and Control of a Continuous-Time Bi-Virus Model},'' \emph{IEEE
  Transactions on Automatic Control}, vol.~64, no.~12, pp. 4891--4906, Dec.
  2019.

\bibitem{goh1978sector}
B.~S. Goh, ``Sector stability of a complex ecosystem model,''
  \emph{Mathematical Biosciences}, vol.~40, no. 1-2, pp. 157--166, 1978.

\bibitem{anderson2017_lefschetzECC_arxiv}
B.~D.~O. Anderson and M.~Ye, ``{Nonlinear Mapping Convergence and Application
  to Social Networks},'' in \emph{European Control Conference, Limassol,
  Cyprus}, Jun. 2018, pp. 557--562.

\bibitem{hirsch2012differential}
M.~W. Hirsch, \emph{{Differential Topology}}.\hskip 1em plus 0.5em minus
  0.4em\relax Springer Science \& Business Media, 2012, vol.~33.

\bibitem{jia2015opinion_SIAM}
P.~Jia, A.~MirTabatabaei, N.~E. Friedkin, and F.~Bullo, ``{Opinion Dynamics and
  the Evolution of Social Power in Influence Networks},'' \emph{SIAM Review},
  vol.~57, no.~3, pp. 367--397, 2015.

\bibitem{ye2020_IFAC_impossible}
M.~Ye, J.~Liu, B.~D.~O. Anderson, and M.~Cao, ``{Distributed Feedback Control
  on the SIS Network Model: An Impossibility Result},'' in \emph{21st IFAC
  World Congress, Berlin}, 2020.

\bibitem{berman1979nonnegative_matrices}
A.~Berman and R.~J. Plemmons, \emph{{Nonnegative Matrices in the Mathematical
  Sciences}}, ser. Computer Science and Applied Mathematics.\hskip 1em plus
  0.5em minus 0.4em\relax Academic Press: London, 1979.

\bibitem{nowzari2016epidemics}
C.~Nowzari, V.~M. Preciado, and G.~J. Pappas, ``{Analysis and Control of
  Epidemics: A Survey of Spreading Processes on Complex Networks},'' \emph{IEEE
  Control Systems}, vol.~36, no.~1, pp. 26--46, 2016.

\bibitem{blanchini1999set_invariance}
F.~Blanchini, ``Set invariance in control,'' \emph{Automatica}, vol.~35,
  no.~11, pp. 1747--1767, 1999.

\bibitem{ye2019_PH_submit}
\BIBentryALTinterwordspacing
M.~Ye, J.~Liu, B.~D.~O. Anderson, and M.~Cao, ``{Applications of the
  Poincar\'e--Hopf Theorem: Epidemic Models and Lotka--Volterra Systems},''
  provisionally accepted in \emph{IEEE Transactions on Automatic Control},
  2020. [Online]. Available: \url{https://arxiv.org/abs/1911.12985}
\BIBentrySTDinterwordspacing

\bibitem{guillemin2010differential}
V.~Guillemin and A.~Pollack, \emph{{Differential Topology}}.\hskip 1em plus
  0.5em minus 0.4em\relax American Mathematical Soc., 2010, vol. 370.

\bibitem{sastry1999nonlinearbook}
S.~Sastry, \emph{Nonlinear systems: analysis, stability, and control}.\hskip
  1em plus 0.5em minus 0.4em\relax Springer New York, 1999, vol.~10.

\bibitem{khamsi2011metric_space_book}
M.~A. Khamsi and W.~A. Kirk, \emph{{An Introduction to Metric Spaces and Fixed
  Point Theory}}.\hskip 1em plus 0.5em minus 0.4em\relax John Wiley \& Sons,
  2011.

\bibitem{jacquez1993qualitative_compartmental}
J.~A. Jacquez and C.~P. Simon, ``{Qualitative Theory of Compartmental
  Systems},'' \emph{SIAM Review}, vol.~35, no.~1, pp. 43--79, 1993.

\bibitem{qu2009cooperative_book}
Z.~Qu, \emph{{Cooperative Control of Dynamical Systems: Applications to
  Autonomous Vehicles}}.\hskip 1em plus 0.5em minus 0.4em\relax Springer
  Science \& Business Media, 2009.

\bibitem{preciado2014epidemic_optimal}
V.~M. Preciado, M.~Zargham, C.~Enyioha, A.~Jadbabaie, and G.~Pappas, ``{Optimal
  Resource Allocation for Network Protection: A Geometric Programming
  Approach},'' \emph{IEEE Transactions on Control of Network Systems}, vol.~1,
  no.~1, pp. 99--108, 2014.

\bibitem{watkins2016optimal_virus}
N.~J. Watkins, C.~Nowzari, V.~M. Preciado, and G.~J. Pappas, ``{Optimal
  Resource Allocation for Competitive Spreading Processes on Bilayer
  Networks},'' \emph{IEEE Transactions on Control of Network Systems}, vol.~5,
  no.~1, pp. 298--307, 2018.

\bibitem{wan2008control_virus}
Y.~Wan, S.~Roy, and A.~Saberi, ``Designing spatially heterogeneous strategies
  for control of virus spread,'' \emph{IET Systems Biology}, vol.~2, no.~4, pp.
  184--201, 2008.

\bibitem{torres2016sparse_spreading}
J.~A. Torres, S.~Roy, and Y.~Wan, ``{Sparse Resource Allocation for Linear
  Network Spread Dynamics},'' \emph{IEEE Transactions on Automatic Control},
  vol.~62, no.~4, pp. 1714--1728, 2016.

\bibitem{mai2018_supress_epidemic}
V.~S. Mai, A.~Battou, and K.~Mills, ``{Distributed Algorithm for Suppressing
  Epidemic Spread in Networks},'' \emph{IEEE Control Systems Letters}, vol.~2,
  no.~3, pp. 555--560, 2018.

\bibitem{hemila2017zinc}
H.~Hemil{\"a}, ``Zinc lozenges and the common cold: a meta-analysis comparing
  zinc acetate and zinc gluconate, and the role of zinc dosage,'' \emph{Journal
  of Royal Society of Medicine Open}, vol.~8, no.~5, p. 2054270417694291, 2017.

\bibitem{wang2004epidemic_patchy}
W.~Wang and X.-Q. Zhao, ``An epidemic model in a patchy environment,''
  \emph{Mathematical Biosciences}, vol. 190, no.~1, pp. 97--112, 2004.

\bibitem{jin2005effect_patchy}
Y.~Jin and W.~Wang, ``The effect of population dispersal on the spread of a
  disease,'' \emph{Journal of Mathematical Analysis and Applications}, vol.
  308, no.~1, pp. 343--364, 2005.

\bibitem{li2009patchy_SIR}
M.~Y. Li and Z.~Shuai, ``Global stability of an epidemic model in a patchy
  environment,'' \emph{Canadian Applied Mathematics Quarterly}, vol.~17, no.~1,
  pp. 175--187, 2009.

\bibitem{vidyasagar2002nonlinear}
M.~Vidyasagar, \emph{{Nonlinear Systems Analysis}}.\hskip 1em plus 0.5em minus
  0.4em\relax SIAM, 2002, vol.~42.

\bibitem{walter1971ordinary}
W.~Walter, ``{Ordinary Differential Inequalities in Ordered Banach Spaces},''
  \emph{Journal of Differential Equations}, vol.~9, no.~2, pp. 253--261, 1971.

\bibitem{allen1994DT_epidemics}
L.~J.~S. Allen, ``{Some Discrete-Time $SI$, $SIR$, and $SIS$ Eepidemic
  Models},'' \emph{Mathematical biosciences}, vol. 124, no.~1, pp. 83--105,
  1994.

\bibitem{prem2020covid}
K.~Prem, Y.~Liu, T.~W. Russell, A.~J. Kucharski, R.~M. Eggo, N.~Davies,
  S.~Flasche, S.~Clifford, C.~A. Pearson, J.~D. Munday \emph{et~al.}, ``{The
  effect of control strategies to reduce social mixing on outcomes of the
  COVID-19 epidemic in Wuhan, China: a modelling study},'' \emph{The Lancet
  Public Health}, vol.~5, no.~5, pp. 261--270, May 2020.

\bibitem{din2013DT_lotka}
Q.~Din, ``{Dynamics of a discrete Lotka-Volterra model},'' \emph{Advances in
  Difference Equations}, vol. 2013, no.~1, p.~95, 2013.

\bibitem{ye2019DF_journal}
M.~Ye, J.~Liu, B.~D.~O. Anderson, C.~Yu, and T.~Ba{\c{s}}ar, ``{Evolution of
  Social Power in Social Networks with Dynamic Topology},'' \emph{IEEE
  Transaction on Automatic Control}, vol.~63, no.~11, pp. 3793--3808, Nov.
  2018.

\bibitem{armstrong2013basic}
M.~A. Armstrong, \emph{{Basic Topology}}.\hskip 1em plus 0.5em minus
  0.4em\relax Springer Science \& Business Media, 2013.

\bibitem{ye2019_ECC_DF_Distort}
M.~Ye and B.~D.~O. Anderson, ``{Modelling of Individual Behaviour in the
  DeGroot--Friedkin Self-Appraisal Dynamics on Social Networks},'' in
  \emph{European Control Conference, Naples, Italy}, Jun. 2019, pp. 2011--2017.

\bibitem{ji1994global_cooperative}
J.~Ji-Fa, ``On the global stability of cooperative systems,'' \emph{Bulletin of
  the London Mathematical Society}, vol.~26, no.~5, pp. 455--458, 1994.

\end{thebibliography}
%


\begin{IEEEbiography}
[{\includegraphics[width=1in,height=1.25in,clip,keepaspectratio]{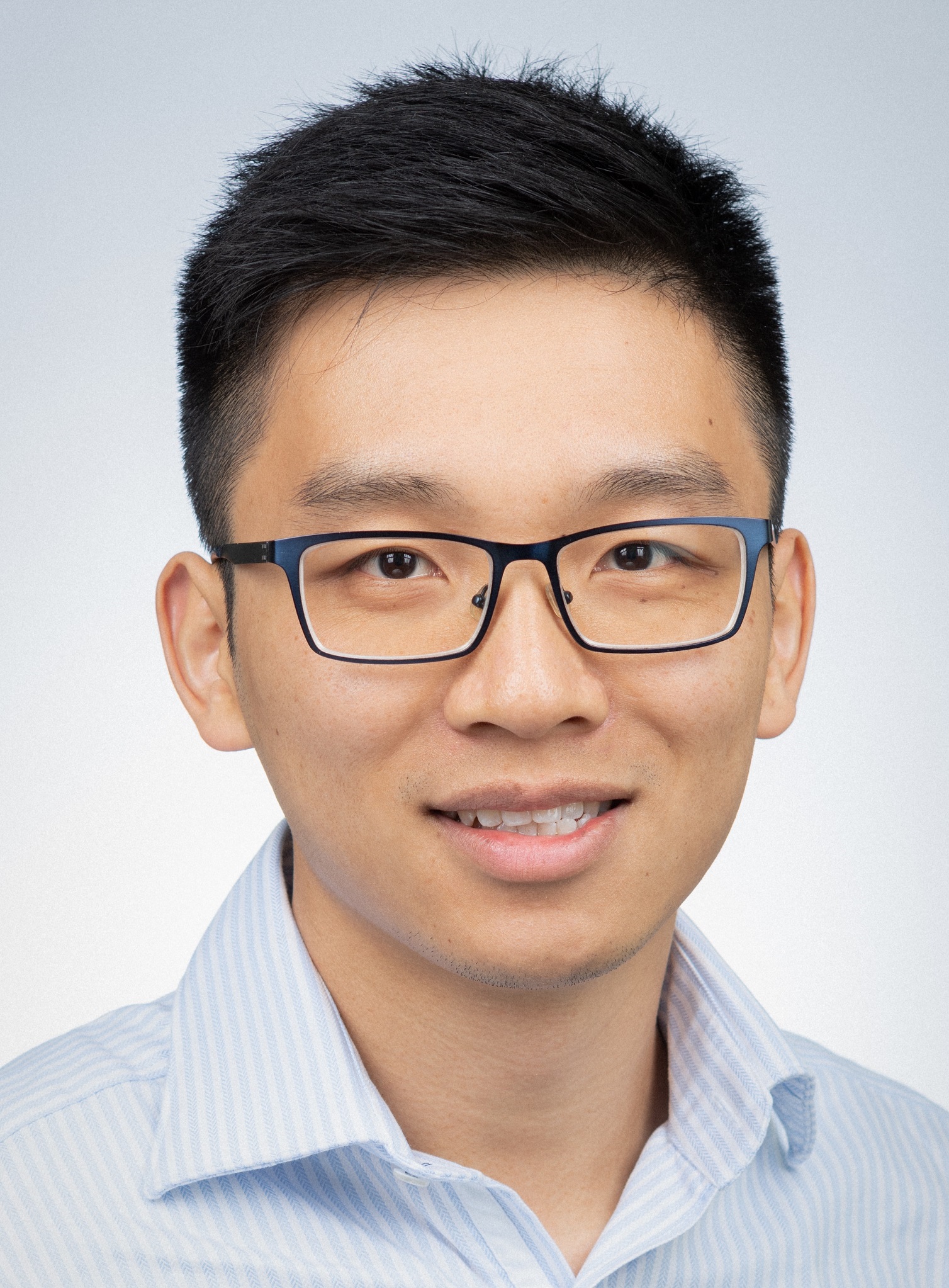}}]{Mengbin Ye} (S'13-M'18)
was born in Guangzhou, China. He received the B.E. degree (with First Class Honours) in mechanical engineering from the University of Auckland, Auckland, New Zealand in 2013, and the Ph.D. degree in engineering at the Australian National University, Canberra, Australia in 2018. From 2018-2020, he was a postdoctoral researcher with the Faculty of Science and Engineering, University of Groningen, Netherlands. Since 2020, he is an Optus Fellow at the Optus--Curtin Centre of Excellence in Artificial Intelligence, Curtin University, Perth, Australia.
	
He has been awarded the 2018 J.G. Crawford Prize (Interdisciplinary), ANU's premier award recognising graduate research excellence. He has also received the 2018 Springer PhD Thesis Prize, and was Highly Commended in the Best Student Paper Award at the 2016 Australian Control Conference. His current research interests include opinion dynamics and decision making in complex social networks, epidemic modelling, coordination of multi-agent systems, and localisation using bearing measurements.
\end{IEEEbiography}

\begin{IEEEbiography}
[{\includegraphics[width=1in,height=1.25in,clip,keepaspectratio]{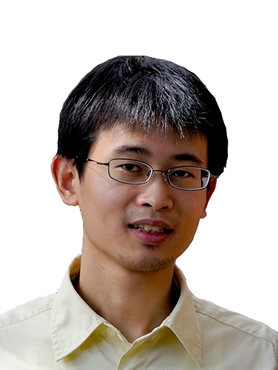}}]{Ji Liu} (S'09-M'13) received the B.S. degree in information engineering from Shanghai Jiao Tong University, Shanghai, China, in 2006, 
and the Ph.D. degree in electrical engineering from Yale University, New Haven, CT, USA, in 2013. He is currently an 
Assistant Professor in the Department of Electrical and Computer Engineering at Stony Brook University, Stony Brook, NY, 
USA. Prior to joining Stony Brook University, he was a Postdoctoral Research Associate at the Coordinated Science 
Laboratory, University of Illinois at Urbana-Champaign, Urbana, IL, USA, and the School of Electrical, Computer and Energy 
Engineering, Arizona State University, Tempe, AZ, USA. His current research interests include distributed control and
computation, distributed optimization and learning, multi-agent systems, social networks, epidemic networks, and 
cyber-physical systems.
\end{IEEEbiography}

\begin{IEEEbiography}
[{\includegraphics[width=1in,height=1.25in,clip,keepaspectratio]{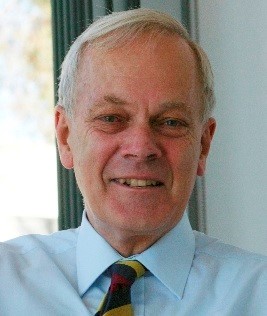}}]{Brian D.O. Anderson} (M'66-SM'74-F'75-LF'07) was born in Sydney, Australia. He received the B.Sc. degree in pure mathematics in 1962, and B.E. in electrical engineering in 1964, from the Sydney University, Sydney, Australia, and the Ph.D. degree in electrical engineering from Stanford University, Stanford, CA, USA, in 1966.

He is an Emeritus Professor at the Australian National University, and a Distinguished Researcher in Data61-CSIRO (previously NICTA) and a Distinguished Professor at Hangzhou Dianzi University. His awards include the IEEE Control Systems Award of 1997, the 2001 IEEE James H Mulligan, Jr Education Medal, and the Bode Prize of the IEEE Control System Society in 1992, as well as several IEEE and other best paper prizes. He is a Fellow of the Australian Academy of Science, the Australian Academy of Technological Sciences and Engineering, the Royal Society, and a foreign member of the US National Academy of Engineering. He holds honorary doctorates from a number of universities, including Universit\'{e} Catholique de Louvain, Belgium, and ETH, Z\"{u}rich. He is a past president of the International Federation of Automatic Control and the Australian Academy of Science. His current research interests are in distributed control, sensor networks and econometric modelling. 
\end{IEEEbiography}

\begin{IEEEbiography}
[{\includegraphics[width=1in,height=1.25in,clip,keepaspectratio]{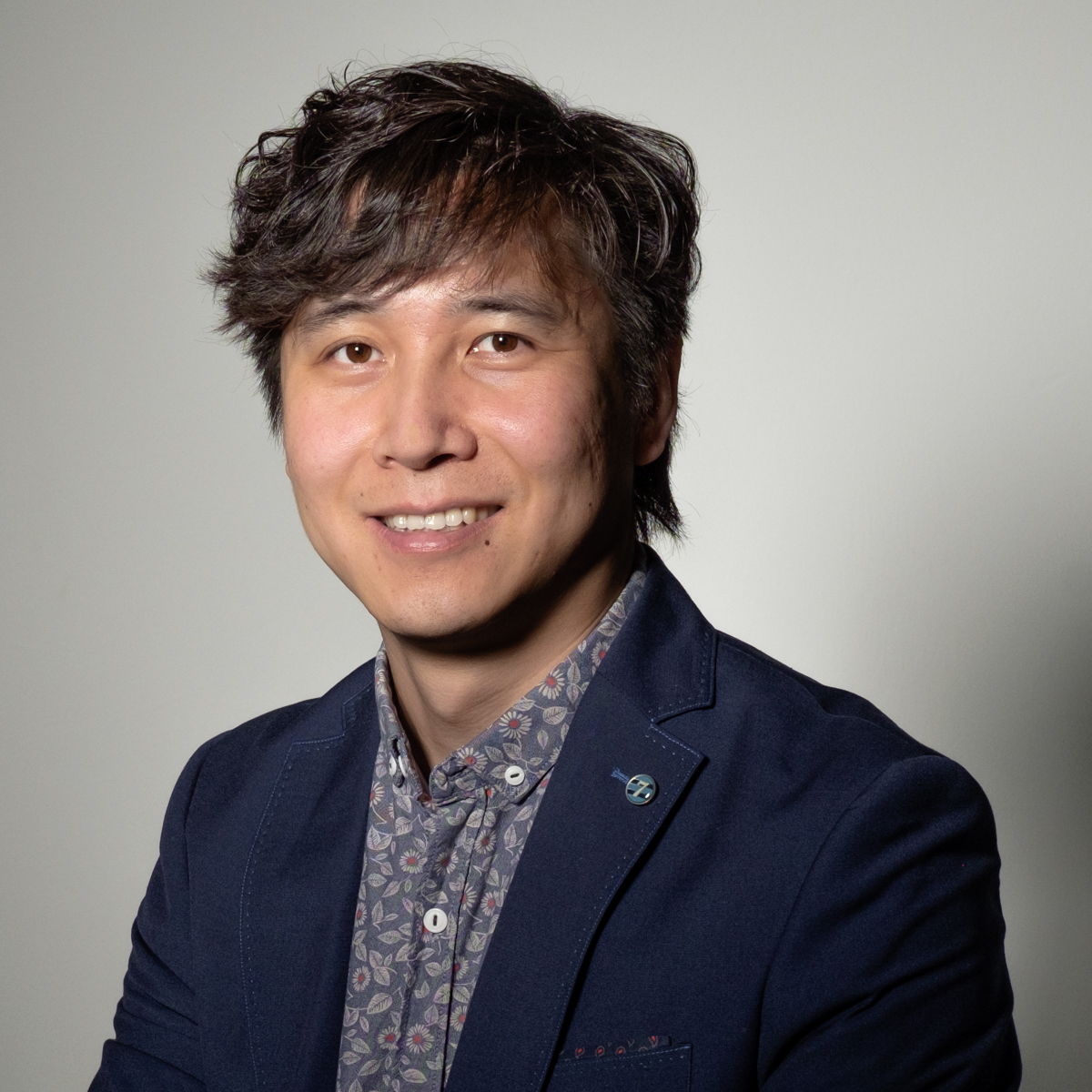}}]{Ming Cao} is  currently  Professor  of  systems  and control  with  the  Engineering  and  Technology  Institute  (ENTEG)  at  the  University  of  Groningen, the  Netherlands,  where  he  started  as  a  tenure-track Assistant Professor in 2008. He received the Bachelor degree in 1999 and the Master degree in 2002
from  Tsinghua  University,  Beijing,  China,  and  the
Ph.D.  degree  in  2007  from  Yale  University,  New
Haven, CT, USA, all in Electrical Engineering. From
September  2007  to  August  2008,  he  was  a  Postdoctoral Research Associate with the Department of Mechanical  and  Aerospace  Engineering  at  Princeton  University,  Princeton, NJ,  USA.  He  worked  as  a  research  intern  during  the  summer  of  2006  with the  Mathematical  Sciences  Department  at  the  IBM  T.  J.  Watson  Research Center,  NY,  USA. 

He  is  the  2017  and  inaugural  recipient  of  the  Manfred Thoma medal from the International Federation of Automatic Control (IFAC) and  the  2016  recipient  of  the  European  Control  Award  sponsored  by  the European  Control Association  (EUCA).  He  is  an  Associate  Editor  for  IEEE
Transactions   on   Automatic   Control,   IEEE   Transactions   on   Circuits   and Systems  and  Systems  and  Control  Letters,  and  for  the  Conference  Editorial Board  of  the  IEEE  Control  Systems  Society.  He  is  also  a  member  of  the IFAC  Technical  Committee  on  Networked  Systems.  His  research  interests include autonomous agents and multi-agent systems, mobile sensor networks and complex networks.
\end{IEEEbiography}




\end{document}